\documentclass[letter,final,oneside,notitlepage,10pt]{article}

\usepackage{amssymb}
\usepackage{amsmath}
\usepackage{amstext}
\usepackage{mathrsfs}
\usepackage[all]{xy}
\usepackage{xstring}
\usepackage{amsthm}
\usepackage[]{geometry}
\usepackage{graphicx}
\usepackage[margin=2cm,font=normalsize,labelfont=bf]{caption}
\usepackage{verbatim}
\usepackage{enumerate}
\usepackage{fancyhdr}
\usepackage[normalem]{ulem}
\usepackage{hyperref}
\usepackage[latin1]{inputenc}
\usepackage[english]{babel}
\usepackage{rotating}

\def\Z {\mathbb{Z}}

\def\C {\mathbb{C}}
\def\im{\mathrm{i}}
\def\id{\mathrm{id}}
\def\h {\check{\mathrm{H}}}

\def\trivlin{\mathbf{I}}

\def\quand{\quad\text{ and }\quad}
\def\quomma{\quad\text{, }\quad}

\def\nobr{~\hspace{-0.26em}}
\def\maps{\nobr:\nobr}
\def\df{\nobr := \nobr}
\def\eq{\nobr = \nobr}

\def\lift{\mathcal{L}\!i\!f\!t}

\def\lli#1{\,_{#1}\!}

\def\upi{\underline{\pi}}
\def\class{cl\!}

\def\idmorph#1{#1_{dis}}
\renewcommand{\varepsilon}{\epsilon}
\def\bigset#1#2{\left\lbrace\;\begin{minipage}[c]{#1}\begin{center}#2\end{center}\end{minipage}\;\right\rbrace}

\newcommand\erf[1]{(\ref{#1})}

\newlength{\myl}

\newcommand{\ueins}{{\mathrm{U}}(1)}

\newcommand{\spin}[1]{{\mathrm{Spin}}(#1)}

\newcommand{\str}[1]{{\mathrm{String}}(#1)}

\def\des{\mathcal{D}\!esc}

\def\triv{\mathcal{T}\!\!riv}

\def\calg {\mathcal{G}}

\def\brackets#1{\IfStrEq{#1}{-}{}{(#1)}}
\def\buntech#1#2{\mathcal{B}\hspace{-0.01em}un^{#2}_{\hspace{-0.04em}#1}}
\def\bun#1#2{\buntech{#1}{}\brackets{#2}}
\def\zwoabun#1#2{2\text{\text{-}}\buntech{#1}{}\brackets{#2}}

\def\grbtech#1{\mathcal{G}\hspace{-0.06em}r\hspace{-0.06em}b_{\hspace{-0.07em}#1}}
\def\grb#1#2{\grbtech#1\brackets{#2}}
\def\grbor#1#2#3{\grbtech#1\brackets{#2}^{#3\text{-}or}}
\def\grbcon#1#2{\grbtech{#1}^{\nabla\!}\brackets{#2}}

\def\Desc {\mathcal{D}\!esc}

\def\Desc{\mathcal{D}\!e\!s\!c}

\newcommand{\alxydim}[2]{\begin{aligned}\xymatrix#1{#2}\end{aligned}}
\renewcommand{\to}{\!\xymatrix@R=0cm@C=1.4em{\ar[r] &}}
\renewcommand{\mapsto}{\!\xymatrix@R=0cm@C=1.4em{\ar@{|->}[r] &}\!}
\renewcommand{\Rightarrow}{\!\xymatrix@R=0cm@C=1.4em{\ar@{=>}[r] &}\!}
\renewcommand{\Leftarrow}{\!\xymatrix@R=0cm@C=1.4em{\ar@{<=}[r] &}\!}
\newcommand{\incl}{\!\xymatrix@R=0cm@C=1.4em{\ar@{^(->}[r] &}\!}

\renewcommand\Leftrightarrow{\!\xymatrix@R=0cm@C=1.4em{\ar@{<=>}[r] &}\!}

\makeatletter

\newcounter{denseversion}
\newcounter{authorcounter}
\newcounter{adresscounter}

\def\title#1{\gdef\@title{#1}}
\def\@title{}
\def\subtitle#1{\gdef\@subtitle{#1}}
\def\@subtitle{}

\def\authortagsused{0}
\def\adresstag#1{\if!#1!\else$^{\;#1\;}$\fi}

\renewcommand{\author}[2][]{
  \stepcounter{authorcounter}
  \if!#1!\else\gdef\authortagsused{1}\fi
  \ifnum\value{authorcounter}=1
    \def\@authorstringa{#2\adresstag{#1}}
    \def\@authorstringb{#2}
    \def\@authorstringc{#2\adresstag{#1}}
  \else
    \g@addto@macro\@authorstringa{\ and #2\adresstag{#1}}
    \g@addto@macro\@authorstringb{\ and #2}
    \g@addto@macro\@authorstringc{\\#2\adresstag{#1}}
  \fi}
\def\@author{\ifnum\value{denseversion}=0\@authorstringa\else\@authorstringb\fi}

\def\@adressstringa{}
\def\@adressstringb{}
\newcommand{\adress}[2][]{
  \stepcounter{adresscounter}
  \ifnum\value{adresscounter}=1
    \g@addto@macro\@adressstringa{\ifnum\authortagsused=0\def\br{\\}\else\def\br{, }\fi\adresstag{#1}#2}
    \g@addto@macro\@adressstringb{\def\br{\\}\adresstag{#1}\parbox[t]{14cm}{#2}}
  \else
    \g@addto@macro\@adressstringa{\\[\bigskipamount]\adresstag{#1}#2}
    \g@addto@macro\@adressstringb{\\[\medskipamount]\adresstag{#1}\parbox[t]{14cm}{#2}}
  \fi}
\def\@adress{\ifnum\value{denseversion}=0\@adressstringa\else\@adressstringb\fi}

\def\preprint#1{\gdef\@preprint{#1}}
\def\@preprint{}
\def\keywords#1{\gdef\@keywords{#1}}
\def\@keywords{}
\def\msc#1{\gdef\@msc{#1}}
\def\@msc{}
\def\email#1{
   \gdef\@email{#1}
   \g@addto@macro\@authorstringc{ {\it (#1)}}}
\def\@email{}
\def\dedication#1{\gdef\@dedication{#1}}
\def\@dedication{}

\def\mybaselinestretch#1{\gdef\@mybaselinestretch{#1}}
\def\@mybaselinestretch{}

\def\refname{References}

\newlength{\myparskip}
\newlength{\myproofparskip}

\mybaselinestretch{1.2}
\renewcommand{\baselinestretch}{\@mybaselinestretch}

\def\denseversion{
  \setcounter{denseversion}{1}
  \newgeometry{left=2cm,right=2cm,top=2cm}
  \mybaselinestretch{1.1}
  \renewcommand{\baselinestretch}{\@mybaselinestretch}
  \normalfont
  \fancyfoot[C]{\itshape{\hspace{2.5cm}--$\,\,$\thepage$\,\,$--}}}

\renewcommand{\emph}[1]{\def\reserved@a{it}\ifx\f@shape\reserved@a\uline{#1}\else\textit{#1}\fi}

\newcommand{\mytableofcontents}{
   \ifnum\value{denseversion}=0
     \tableofcontents
   \else
     \renewcommand{\baselinestretch}{0.8}
     \normalfont
     \tableofcontents
     \renewcommand{\baselinestretch}{\@mybaselinestretch}
     \normalfont
   \fi}

\newcounter{mythm}[subsection]
\newcounter{mainthm}

\def\setsecnumdepth#1{
  \setcounter{secnumdepth}{#1}
  \setcounter{mythm}{0}
  \ifnum \c@secnumdepth >0
    \ifnum \c@secnumdepth >1
      \def\themythm{\thesubsection.\arabic{mythm}}
      \numberwithin{equation}{subsection}
      \renewcommand\theequation{\thesubsection.\arabic{equation}}
    \else
      \def\themythm{\thesection.\arabic{mythm}}
      \numberwithin{equation}{section}
      \renewcommand\theequation{\thesection.\arabic{equation}}
    \fi
  \else
    \def\themythm{\arabic{mythm}}
  \fi}

\newenvironment{mythmenv}{\strut\ \setlength{\parskip}{\myproofparskip}}{\setlength{\parskip}{\myparskip}}

\newlength{\mythmskip}
\newlength{\mythmtopskip}
\newtheoremstyle{mythmstylea}{\mythmtopskip}{\mythmskip}{\it}{}{\bf}{.}{0em}{}
\newtheoremstyle{mythmstyleb}{\mythmtopskip}{\mythmskip}{}{}{\bf}{.}{0em}{}

\theoremstyle{mythmstylea}
\newtheorem{mytheorem}[mythm]{Theorem}
\newtheorem{mydefinition}[mythm]{Definition}
\newtheorem{mycorollary}[mythm]{Corollary}
\newtheorem{myproposition}[mythm]{Proposition}
\newtheorem{mylemma}[mythm]{Lemma}

\newtheorem{mymaintheorem}[mainthm]{Theorem}
\newtheorem{mymaincorollary}[mainthm]{Corollary}
\newtheorem{mymaindefinition}[mainthm]{Definition}

\theoremstyle{mythmstyleb}
\newtheorem{myremark}[mythm]{Remark}
\newtheorem{myexample}[mythm]{Example}

\newenvironment{theorem}[1][]{\begin{mytheorem}[#1]\begin{mythmenv}}{\end{mythmenv}\end{mytheorem}}
\newenvironment{definition}[1][]{\begin{mydefinition}[#1]\begin{mythmenv}}{\end{mythmenv}\end{mydefinition}}
\newenvironment{corollary}[1][]{\begin{mycorollary}[#1]\begin{mythmenv}}{\end{mythmenv}\end{mycorollary}}
\newenvironment{proposition}[1][]{\begin{myproposition}[#1]\begin{mythmenv}}{\end{mythmenv}\end{myproposition}}
\newenvironment{lemma}[1][]{\begin{mylemma}[#1]\begin{mythmenv}}{\end{mythmenv}\end{mylemma}}
\newenvironment{remark}[1][]{\begin{myremark}[#1]\begin{mythmenv}}{\end{mythmenv}\end{myremark}}
\newenvironment{example}[1][]{\begin{myexample}[#1]\begin{mythmenv}}{\end{mythmenv}\end{myexample}}

\newenvironment{maintheorem}[1]{\def\themainthm{#1}\begin{mymaintheorem}\begin{mythmenv}}{\end{mythmenv}\end{mymaintheorem}}

\renewenvironment{proof}[1][Proof]{\noindent #1. \begin{mythmenv}}{\ \strut\hfill{$\square$}\end{mythmenv}\medskip}

\def\maketitle{
  \setlength{\parskip}{\myparskip}  
  \newpage
  \noindent
  \begin{center}
    \LARGE\@title\\
    \if!\@subtitle!\else \smallskip\LARGE\@subtitle\\\fi
    \bigskip
    \if!\@author!\else\bigskip\large\@author\\\fi
    \ifnum\value{denseversion}=0
      \if!\@adress!\else     \bigskip\normalsize\@adress\\\fi
      \if!\@email!\else\ifnum\value{authorcounter}=1\bigskip\normalsize\textit{\@email}\\\else\fi\fi
    \else
    \fi
    \if!\@dedication!\else \bigskip\normalsize{\@dedication}\\\fi
  \end{center}
  \ifnum\value{denseversion}=0\vskip 1.5cm\else\vskip0.5cm\fi
  \thispagestyle{empty}}

\def\br{
  \ifnum\value{denseversion}=0
  \\
  \else\fi
}

\def\showkeywords{\begin{flushleft}\footnotesize\textbf{Keywords}: \@keywords\end{flushleft}}
\def\showmsc{\begin{flushleft}\footnotesize\textbf{MSC 2010}: \@msc\end{flushleft}}

\def\kobiburl#1{
   \IfBeginWith
     {#1}
     {http://arxiv.org/abs/}
     {\kobibarxiv{#1}}
     {\kobiblink{#1}}}
\def\kobibarxiv#1{\href{#1}{\texttt{[arxiv:\StrGobbleLeft{#1}{21}]}}}
\def\kobiblink#1{Available as: \href{#1}{\texttt{#1}}}

\newcommand{\etalchar}[1]{$^{#1}$}
\def\kobib#1{
  \begin{raggedright}
  \ifnum\value{denseversion}=0\else\small\fi

  \end{raggedright}
  \noindent
  \ifnum\value{denseversion}=0\else
      \if!\@authorstringc!\else
        \ifnum\authortagsused=0\ifnum\value{authorcounter}>1\normalsize\@authorstringc\\[\medskipamount]\else\fi\else\normalsize\@authorstringc\\[\medskipamount]\fi       \fi
      \if!\@adress!\else\normalsize\@adress\\\fi
      \ifnum\authortagsused=0\ifnum\value{authorcounter}=1\if!\@email!\else\normalsize\textit{\@email}\\\fi\else\fi\else\fi
  \fi}

\newenvironment{commentfigure}{}
\newenvironment{sidewayscommentfigure}{\begin{minipage}}{\end{minipage}}

\def\showcomments{ -- Comments suppressed}

\newif\if@fewtab\@fewtabtrue{
  \count255=\time\divide\count255 by 60
  \xdef\hourmin{\number\count255}
  \multiply\count255 by-60\advance\count255 by\time
  \xdef\hourmin{\hourmin:\ifnum\count255<10 0\fi\the\count255}}
\def\ps@draft{
  \let\@mkboth\@gobbletwo
  \def\@oddfoot{
    \hbox to 7 cm{\tiny \versionno\hfil}
    \hskip -7cm\hfil\rm\thepage\hfil{\tiny\draftdate}}
  \def\@oddhead{}
  \def\@evenhead{}
  \let\@evenfoot\@oddfoot}
\def\draftdate{\number\month/\number\day/\number\year\ \ \ \hourmin }
\newcommand\version[1]{
  \typeout{}\typeout{#1}\typeout{}
  \vskip-1.7cm \centerline{\fbox{{\normalsize\tt DRAFT -- #1 -- 
  \draftdate\showcomments}}} \vskip0.92cm}
\def\draft#1{
  \def\versionno{#1}
  \pagestyle{draft}\thispagestyle{draft}
  \gdef\@ntitle{\version\versionno \@title}
  \global\def\draftcontrol{1}}
\global\def\draftcontrol{0}

\makeatother

\usepackage[latin1]{inputenc}
\usepackage[english]{babel}

\def\quot#1{``#1''}

\title{Lifting Problems and Transgression for Non-Abelian Gerbes}

\author[a]{Thomas Nikolaus}
\email{thomas1.nikolaus@mathematik.uni-regensburg.de}

\author[a,b]{Konrad Waldorf}
\email{konrad.waldorf@mathematik.uni-regensburg.de}

\adress[a]{Fakultät für Mathematik\br Universität Regensburg\\Universitätsstraße 31\br D-93053 Regensburg}

\adress[b]{Hausdorff Research Institute for Mathematics\\Poppeldorfer Allee 45\br D-53115 Bonn}

\msc{Primary 55R65, Secondary 18G50, 22A22, 53C08, 55N05}

\usepackage[normalem]{ulem}
\usepackage{amstext}

\newcommand\leer[1]{}

\begin{document}

\denseversion

\setsecnumdepth{3}

\maketitle

\begin{abstract}
We discuss lifting and reduction problems for bundles and gerbes in the context of a Lie 2-group. We obtain a geometrical formulation (and a new proof) for the exactness of Breen's long exact sequence in non-abelian cohomology. We use our geometrical formulation in order to define a transgression map in non-abelian cohomology. This transgression map relates the degree one non-abelian cohomology of a smooth manifold (represented by non-abelian gerbes) with the degree zero non-abelian cohomology of the free loop space  (represented by principal bundles). We prove several properties for this transgression map. For instance,  it reduces -- in case of a Lie 2-group with a single object -- to the ordinary transgression  in ordinary cohomology. We describe applications  of our results to string manifolds: first, we obtain a new  comparison theorem for different notions of string structures. Second,  our  transgression map establishes a direct relation between string structures and spin structures on the loop space.
\showkeywords
\showmsc
\end{abstract}


\mytableofcontents

\setsecnumdepth{1}

\section{Introduction}

\label{sec:intro}

In the present paper we study categorical and bicategorical extensions of the non-abelian cohomology of a smooth manifold $M$. In degree zero, non-abelian cohomology with values in a Lie groupoid $\Gamma$ is a set denoted $\h^0(M,\Gamma)$, in which a cocycle with respect to a cover of $M$ by open sets $U_i$
is a collection of smooth maps
\begin{equation*}
\alpha_i\maps U_i \to \Gamma_0
\quand
g_{ij}\maps U_i \cap U_j \to \Gamma_1
\end{equation*}
subject to a cocycle condition.    The categorical extension (or geometric model) of the set $\h^0(M,\Gamma)$ is the groupoid $\bun\Gamma M$ of \emph{principal $\Gamma$-bundles} over $M$. 
Such bundles have a total space $p\maps P  \to M$, an anchor map $\alpha\maps P  \to \Gamma_0$, and carry a principal \quot{action of $\Gamma$ along $\alpha$}. 
The relation between the set $\h^0(M,\Gamma)$ and the groupoid $\bun\Gamma M$ is that the first is the set of isomorphism classes of the second; see \cite[Section 5.7]{moerdijk} or 
\cite[Section 2.2]{NW11}.

If $G$ is a Lie group, there is a  Lie groupoid $\mathcal{B}G$ with a single object and the group $G$ as its morphisms. A principal $\mathcal{B}G$-bundle is the same as an ordinary principal $G$-bundle. Correspondingly, $\h^0(M,\mathcal{B}G)$ is the classical \v Cech cohomology $\h_{\class}^1(M,G)$.

In this paper we want to look at \emph{degree one}  non-abelian cohomology. This requires additional structure on the Lie groupoid $\Gamma$: a Lie 2-group structure -- a certain kind of monoidal structure. 
Degree one non-abelian cohomology with values in a Lie 2-group $\Gamma$ is a set denoted by $\h^1(M,\Gamma)$, and a cocycle is a collection of smooth maps
\begin{equation*}
g_{ij}\maps U_i \cap U_j \to \Gamma_0
\quand
f_{ijk}\maps U_i \cap U_{j} \cap U_k \to \Gamma_1
\end{equation*} 
subject to higher cocycle conditions. The bicategorical extension (resp. geometric model) of the set  $\h^1(M,\Gamma)$ is the bigroupoid
$\zwoabun\Gamma M$ of \emph{principal $\Gamma$-2-bundles} over $M$. The total space of such a 2-bundle is a Lie groupoid $\mathcal{P}$ that carries a principal action of the Lie 2-group $\Gamma$. 
The set of isomorphism classes of the bigroupoid $\zwoabun\Gamma M$ is in bijection with $\h^1(M,\Gamma)$. A detailed account of principal 2-bundles is given in  \cite[Section 6]{NW11}.
In Section \ref{sec:gerbesand2bundles} of the present paper we review some important definitions and results. 

A Lie 2-group $\Gamma$ has two interesting invariants denoted by $\upi_0\Gamma$ and $\upi_1\Gamma$. The first is the set of isomorphism classes of objects of $\Gamma$, and the second is the automorphism group 
of the tensor unit $1 \in \Gamma_0$. For certain Lie 2-groups, $\upi_0\Gamma$ and $\upi_1\Gamma$ are again Lie groups (see Definition \ref{def:ss}). 
For an ordinary topological group $H$ there is a well known long exact sequence relating cohomology (resp. homotopy classes of maps) with values in $H$, $\pi_0(H)$ and $\pi_1(H)$. For a Lie 2-group there is  an  analogue due to Breen \cite{breen3},
which combines the non-abelian cohomology groups $\h^n(M,\Gamma)$ together with the classical \v Cech cohomology groups in a long  exact sequence. With the abbreviation $G \df \upi_0\Gamma$ and $A \df \upi_1\Gamma$, this sequence is:
\begin{multline*}
\hspace{-0.2cm}\alxydim{@C=0.7cm}{0  \ar[r]^{\rule[0mm]{0mm}{2mm}} & \h^1_{\class}(M,A) \ar[r] & \h^0(M,\Gamma) \ar[r] & \h^0_{\class}(M,G) \ar[r] & \h_{\class}^2(M,A)}
\br\hspace{-0.25cm}
\alxydim{@C=0.7cm}{\ar[r]^{\rule[0mm]{0mm}{2mm}} &\h^1(M,\Gamma) \ar[r]^-{\pi_{*}} & \h^1_{\class}(M,G) \ar[r]^-{\delta} & \h^3_{\class}(M,A)\text{.}}
\end{multline*}
In Section \ref{sec:long} we explain this sequence in more detail; in particular we explain how it can be regarded as being induced by a short exact sequence of Lie 2-groups. 

%

Exactness at $\h^1_{\class}(M,G)$ means that for a given principal $G$-bundle $E$ over $M$ the class $\delta([E])$ vanishes if and only if there exists a principal $\Gamma$-2-bundle $\mathcal{P}$ over $M$ such that $\pi_{*}([\mathcal{P}])\eq [E]$. Such 2-bundles are called \emph{$\Gamma$-lifts of $E$}, and form a bigroupoid $\lift_\Gamma(E)$. Essential for the categorical extension of the exactness is a geometrical understanding of the obstruction class $\delta([E]) \in \h^3_{\class}(M,A)$. For this purpose we construct an $A$-bundle 2-gerbe $\mathbb{L}_E$ with characteristic class $[\mathbb{L}_E]\eq \delta([E])$ (Definition \ref{def:lifting2gerbe}).  Our construction of $\mathbb{L}_E$ generalizes the  Chern-Simons bundle 2-gerbe of Carey et al. \cite{carey4} using a new relation between Lie 2-groups and multiplicative gerbes that we discover. The categorical extension of the exactness is now given by the following theorem.

\begin{maintheorem}{A}
\label{th:A}
Let $\Gamma$ be a smoothly separable Lie 2-group, let $E$ be a principal  $G$-bundle over $M$, and let $\mathbb{L}_E$ be the associated lifting bundle 2-gerbe.  Then, there is an equivalence of bigroupoids
\begin{equation*}
\bigset{3.8cm}{Trivializations of the lifting bundle 2-gerbe $\mathbb{L}_E$} \cong 
\lift_{\Gamma}(E)\text{.}
\end{equation*}
\end{maintheorem}

Theorem \ref{th:A} reproduces the set-theoretical exactness statement in the sense that equivalent bigroupoids are either both empty or both non-empty. On top of that, Theorem \ref{th:A} specifies \emph{how} the various ways of trivializing the obstruction are related to the various possible lifts. Theorem \ref{th:A} is stated and proved in the main text as Theorem \ref{thm:lift}. The proof uses the descent theory for bundle gerbes developed in \cite{nikolaus2}, as well as a  reduction theorem (Theorem \ref{reduction}) that establishes a categorical extension of the exactness of the sequence one position to the left.

The main objective of this paper is to define and study non-abelian transgression. For this purpose, the improvement we have achieved with  Theorem \ref{th:A}  will be  \emph{essential}.
 We recall that there is a transgression homomorphism
\begin{equation}
\label{eq:classtrans}
\h^n_{\class}(M,A) \to \h^{n-1}_{\class}(LM,A) 
\end{equation}
in classical \v Cech cohomology, where $LM\eq C^{\infty}(S^1,M)$ is the free loop space of $M$. It is defined in the first place for a differential extension, e.g., Deligne cohomology. There, transgression is a chain map between the Deligne cochain complexes \cite{brylinski1,gomi2};  this chain map induces a well-defined map in the ordinary (non-differential) cohomology.

In a paper \cite{schreiber2} of  Urs Schreiber and KW, a generalization of transgression to non-abelian cohomology using connections on non-abelian gerbes and their formulation by  parallel transport 2-functors is discussed. Unfortunately, the method developed there works only for the based loop space $\Omega M \subset LM$.
Yet, it showed already an important difference between the abelian transgression \erf{eq:classtrans} and its non-abelian generalization: in the non-abelian case, the structure group  changes.

In the formalism of principal 2-bundles, and with the help of the lifting theory of Theorem \ref{th:A}, we are able to resolve the problems encountered in \cite{schreiber2}.  
In Section \ref{sec:transgression} of the present paper we define, for every smoothly separable, strict Lie 2-group $\Gamma$ with $\upi_0\Gamma$ compact, a Fréchet Lie group $L\Gamma$ which we call the \emph{loop group} of $\Gamma$ (Definition \ref{def:loopgroup}). 
This construction uses multiplicative gerbes, connections, and a version of Brylinski's original transgression functor for abelian gerbes \cite{brylinski1}. It is functorial in $\Gamma$, so that a (weak) homomorphism $\Lambda\maps \Gamma \to \Omega$  between Lie 2-groups induces a group homomorphism $L\Lambda\maps L\Gamma \to L\Omega$.

\begin{maintheorem}{B}
\label{th:B}
Let $\Gamma$ be a smoothly separable Lie 2-group with $\upi_0\Gamma$ compact and connected. Then, there is a well-defined transgression map
\begin{equation*}
\h^1(M,\Gamma) \to \h^1_{\class}(LM,L\Gamma)
\end{equation*}
for non-abelian cohomology, which is contravariant in $M$ and covariant in $\Gamma$. Moreover, for $\Gamma\eq\mathcal{B}A$ we have $L\Gamma\eq A$, and the classical transgression map \erf{eq:classtrans} for $n=2$ is reproduced.
\end{maintheorem}

The construction of the transgression map proceeds as follows. First, we represent a non-abelian cohomology class by a principal $\Gamma$-2-bundle $\mathcal{P}$ over $M$. We make the tautological observation that $\mathcal{P}$ is a $\Gamma$-lift of the principal $G$-bundle $E \df \pi_{*}(\mathcal{P})$. By Theorem \ref{th:A}, it thus corresponds to a trivialization of the associated lifting bundle 2-gerbe $\mathbb{L}_E$, which is an \emph{abelian} $A$-bundle 2-gerbe. Second, we use the existing functorial transgression for \emph{abelian} bundle 2-gerbes, resulting in an $A$-bundle gerbe  over $LM$ together with a trivialization. Third, we re-assemble these into a principal $L\Gamma$-bundle over $LM$ using the theory of (ordinary) lifting bundle gerbes \cite{murray}.
 We explain these steps in detail in Section \ref{sec:transgression}, where Theorem \ref{th:B} is stated as Theorem \ref{th:trans}.
The main difficulties we encounter there are to eliminate the choices of connections needed to make abelian transgression functorial.

In Section \ref{sec:string} we present an application of Theorems \ref{th:A} and \ref{th:B} to string structures on a spin manifold $M$. The Lie 2-group which is relevant  here is some strict Lie 2-group model for the string group. This is a (necessarily infinite-dimensional) strict Lie 2-group $\str n$ with   $\upi_0\str n \eq \spin n$ and $\upi_{1}\str n \eq\ueins$, such that its geometric realization is a three-connected  extension  
\begin{equation*}
1 \to B\ueins \to |\str n| \to \spin n \to 1
\end{equation*}
of topological groups. Strict Lie 2-group models have been constructed in \cite{baez9,nikolausb,waldorf14}.
In the language of Theorem \ref{th:A}, we say that a \emph{string structure} on a principal $\spin n$-bundle $E$ is a $\str n$-lift $\mathcal{P}$ of $E$ (Definition \ref{def:ss1}).

We prove that the lifting bundle 2-gerbe $\mathbb{L}_E$ that represents the obstruction against $\str n$-lifts of $E$ coincides with the Chern-Simons 2-gerbe $\mathbb{CS}_E(\mathcal{G})$ associated to the  level one multiplicative bundle gerbe over $G$ (Lemma \ref{lem:levelone}). This bundle 2-gerbe provided another notion of string structures suitable in the context of \emph{string connections} \cite{waldorf8}. Previously, it was known that these two notions of string structures coincide on a level of equivalence classes. Theorem \ref{th:A} promotes this bijection to an  equivalence of  bigroupoids  (Theorem \ref{th:coinc}). It so enables to switch in a \emph{functorial way} between the two notions.   

In order to apply the transgression map of Theorem \ref{th:B} we show that the loop group of $\str n$ is the universal central extension of the loop group  $L\spin n$ (Lemma \ref{lem:loopstring}). Lifts of the structure group of the looped frame bundle of $M$ from $L\spin n$ to this universal extension are usually called \emph{spin structures} on the loop space $LM$ \cite{mclaughlin1}. Previously, it was known that $LM$ is spin if  $M$ is string \cite{mclaughlin1}. Theorem \ref{th:B} now permits to transgress a \emph{specific}  string structure on $M$ to a \emph{specific} spin structure on $LM$ (Theorem \ref{th:stringtrans}).  

\paragraph{Acknowledgements.} KW thanks the Hausdorff Research Institute for Mathematics in Bonn for  kind hospitality and financial support.

\section{Non-Abelian Gerbes and 2-Bundles}

\label{sec:gerbesand2bundles}

In this section we recall the notions of 2-groups,  2-bundles and non-abelian gerbes, and the relationship between the latter two, following our paper \cite{NW11}. 
We remark that these objects have also been treated in many other papers, e.g., \cite{aschieri, breen3, bartels, wockel1}.
There is one difference between \cite{NW11} and the present paper: here we work with  not necessarily finite-dimensional manifolds. More precisely, the manifolds and  Lie groups in this paper are modelled on  locally convex vector spaces.

\begin{definition}
\label{def:lie2group}
A \emph{Lie groupoid} is a  small groupoid $\Gamma$ whose set of objects $\Gamma_0$ and whose set of morphisms $\Gamma_1$ are  manifolds, and whose structure maps
\begin{equation*}
s,t\maps \Gamma_1 \to \Gamma_0 \quad i\maps \Gamma_0 \to \Gamma_1 
\quad \text{and} \quad \circ\maps \Gamma_1 \times_{\Gamma_0} \Gamma_1 \to \Gamma_1 
\end{equation*} 
are smooth, and $s,t$ are submersions. A  \emph{Lie 2-group} is a Lie groupoid $\Gamma$ such that $\Gamma_0$ and $\Gamma_1$ are Lie groups, and such that all structure maps are Lie group homomorphisms.
\end{definition}

We remark that the Lie 2-groups of Definition \ref{def:lie2group} are usually called \emph{strict}. We suppress this adjective since all Lie 2-groups in this article are strict. 
In principle, all theorems and notions discussed here have a counterpart for non-strict 2-groups (even \quot{stacky} 2-groups).

Lie groupoids can be seen as a generalization of Lie groups. Indeed, every Lie group $G$ gives rise to a Lie groupoid $\mathcal{B}G$ with one object, whose manifold  of automorphisms is  $G$. A Lie groupoid of the form $\mathcal{B}G$ is a Lie 2-group if and only if $G$ is abelian.

Let $\Gamma$ be a Lie groupoid, not necessarily a Lie 2-group. Crucial for the present paper is the notion of a \emph{principal $\Gamma$-bundle} over a manifold $M$; see e.g., \cite[Section 5.7]{moerdijk} or \cite[Definition 2.2.1]{NW11}. It generalizes the notion of  a principal bundle for a Lie group. Principal $\Gamma$-bundles over $M$ form  a groupoid $\bun {\Gamma} M$ with many additional features. For example,   one can pullback principal $\Gamma$-bundles along smooth maps $f\maps N \to M$. 
Further, if $\Lambda\maps\Gamma \to \Omega$ is a \emph{smooth anafunctor} between Lie groupoids  \cite[Definition 2.3.1]{NW11}, there is an induced functor 
\begin{equation*}
\Lambda_{*}\maps \bun\Gamma M \to \bun\Omega M
\end{equation*}
which we call \emph{extension} along $\Lambda$ \cite[Definition 2.3.7]{NW11}. This induced functor is an equivalence of groupoids if $\Lambda$ is a \emph{weak equivalence} \cite[Definition 2.3.10 \& Corollary 2.3.11]{NW11}.

If $\Gamma$ is a Lie 2-group, the groupoid of principal $\Gamma$-bundles over $M$ is monoidal by means of a tensor product
\begin{equation*}
\otimes\maps  \bun \Gamma M \times \bun \Gamma M \to \bun \Gamma M.
\end{equation*}
Using this tensor product one can define \emph{non-abelian bundle gerbes} completely analogous to \emph{abelian} bundle gerbes.
\begin{definition}[{{\cite[Definition 5.1.1]{NW11}}}] \label{def:grb}
Let $M$ be a  manifold, and let $\Gamma$ be a Lie 2-group. A \emph{$\Gamma$-bundle gerbe over $M$} is a surjective submersion $\pi\maps Y \to M$, a principal $\Gamma$-bundle $P$ over $Y^{[2]}$ and an associative morphism
\begin{equation*}
\mu\maps \pi_{23}^{*}P \otimes \pi_{12}^{*}P \to \pi_{13}^{*}P
\end{equation*}
of $\Gamma$-bundles over $Y^{[3]}$, where $\pi_{ij}\maps Y^{[3]} \to Y^{[2]}$ denotes the projection to the $i$-th and $j$-th factor.
\end{definition}

Another useful concept of a non-abelian gerbe is a principal 2-bundle. The idea of a principal 2-bundle is to mimic the definition of an ordinary principal bundle, with the structure group replaced by a Lie 2-group $\Gamma$ and the total space replaced by  a Lie groupoid. We will use the notation $\idmorph{M}$ when a manifold $M$ is regarded as a Lie groupoid with only identity morphisms.

\begin{definition}[{{\cite[Definition 6.1.5]{NW11}}}]
A \emph{principal $\Gamma$-2-bundle over $M$} is a Lie groupoid $\mathcal{P}$, a smooth functor $\pi \maps \mathcal{P} \to \idmorph M$ that is a surjective submersion on the level of objects, and a smooth, strict, right action 
\begin{equation*}
R \maps \mathcal{P} \times \Gamma \to \mathcal{P}
\end{equation*}
of $\Gamma$ on $\mathcal{P}$ that is principal in the sense that the functor $(\mathrm{pr}_1, R) \maps \mathcal{P} \times \Gamma \to \mathcal{P} \times_M \mathcal{P} $ is a weak equivalence.
\end{definition}

$\Gamma$-bundle gerbes as well as principal $\Gamma$-2-bundles can be arranged into bigroupoids which we denote by $\grb \Gamma M$ and $\zwoabun \Gamma M$, respectively \cite[Sections 5.1 and 6.1]{NW11}. Moreover, smooth maps $f\maps N \to M$ induce pullback 2-functors for both bigroupoids.

\begin{theorem}[{{\cite[Theorem 7.1]{NW11}}}]
\label{thm:equiv}
The bigroupoids $\grb\Gamma M$ and $\zwoabun \Gamma M$ are equivalent. Furthermore, this equivalence is compatible with pullbacks. More precisely $\grb\Gamma -$ and $\zwoabun \Gamma -$ form equivalent 2-stacks over the site of smooth manifolds. \end{theorem}

Finally, we note that  a  Lie 2-group homomorphism $\Lambda \maps \Gamma \to \Omega$, see \cite[Eq. 2.4.4]{NW11}, induces 2-functors 
\begin{equation*}
\Lambda_{*}\maps \grb\Gamma M \to \grb\Omega M
\quand
\Lambda_{*}\maps\zwoabun \Gamma M \to \zwoabun \Omega M
\end{equation*}
that we call \emph{extension along $\Lambda$}.
These 2-functors are equivalences if the homomorphism $\Lambda$ is a weak equivalence
\cite[Theorems 5.2.2 and 6.2.2]{NW11}.

\section{The Long Exact Sequence in Non-abelian Cohomology}

\label{sec:long}

As mentioned in the Introduction (Section \ref{sec:intro}), the degree zero non-abelian cohomology of a Lie groupoid $\Gamma$  is a  set $\h^0(M,\Gamma)$ that  generalizes the classical \v Cech cohomology. 
If $\Gamma$ is a Lie 2-group, the set $\h^0(M,\Gamma)$ inherits the structure of a (in general non-abelian) group. Further, the Lie 2-group structure permits to define \emph{degree one} non-abelian cohomology,  $\h^1(M,\Gamma)$, as outlined in Section \ref{sec:intro}. Non-abelian cohomology is important for this article because it classifies the geometrical objects of Section \ref{sec:gerbesand2bundles}. 
\begin{theorem}[{{\cite[Theorems 3.3, 5.3.2, 7.1]{NW11}}}]
\label{th:geo}
Let $\Gamma$ be a Lie groupoid and $M$ a paracompact manifold. Then, there is a  bijection 
\begin{equation*}
\h^0(M,\Gamma) \cong \bigset{2.9cm}{Isomorphism classes of principal $\Gamma$-bundles over $M$}\text{.}
\end{equation*}
Moreover, if $\Gamma$ is  a Lie 2-group, there are  bijections
\begin{equation*}
\h^1(M,\Gamma) \cong \bigset{2.9cm}{Isomorphism classes of $\Gamma$-bundle gerbes over $M$} \\
\cong \bigset{3.1cm}{Isomorphism classes of principal $\Gamma$-2-bundles over $M$} \text{.}
\end{equation*}
\end{theorem}

For $\Gamma$ a Lie 2-group, 
we denote by $\upi_{0}\Gamma$ the
 group of  isomorphism classes of objects in $\Gamma$ and by
 $\upi_{1}\Gamma$ the group of automorphisms of $1\in \Gamma_{0}$. We remark that $\upi_{1}\Gamma$ is abelian due to its two commuting group structures (multiplication and composition). We also remark that the sequence
\begin{equation}
\label{4term}
\alxydim{}{1 \ar[r] & \upi_1\Gamma \ar[r]^-{i} & \mathrm{ker}(s) \ar[r]^-{t} & \Gamma_0 \ar[r]^-{\pi} & \upi_0\Gamma \ar[r] & 0}
\end{equation}
is exact as a sequence of groups. 
\begin{definition}
\label{def:ss}
 A Lie 2-group $\Gamma$ is called \underline{smoothly
 separable} if $\upi_{1}\Gamma$ is a split Lie subgroup
 of $\Gamma_{1}$ and the group $\upi_{0}\Gamma$ is a Lie group such that the projection
 $\Gamma_{0}\to \upi_{0}\Gamma$ is a submersion.
\end{definition}

For a smoothly separable Lie 2-group $\Gamma$, the sequence \erf{4term} is exact as a sequence of Lie groups: the inclusion $i$ is an embedding, and the map $t$ is a submersion onto its image. Further, we have the following sequence of Lie 2-groups and smooth functors:
\begin{equation}\label{sequence}
\mathcal{B} \upi_1\Gamma \to \Gamma \to \idmorph{\upi_0\Gamma}\text{.}
\end{equation}
 This sequence is -- in a certain sense -- an exact sequence of Lie 2-groups, although it is not exact on the level of morphisms. More precisely, it is a fibre sequence. We are not going to give a rigorous treatment of fibre sequences in Lie 2-groups, but the reader may keep this in mind as a guiding principle throughout this paper. 


\begin{proposition}[{{\cite[4.2.2]{breen3}}}]
\label{long}
Let $M$ be a paracompact manifold, and let $\Gamma$ be a smoothly separable Lie 2-group. Then, the sequence \eqref{sequence} induces a long exact sequence in non-abelian cohomology:
\begin{equation*}
\xymatrix{
 0 \ar[r] &
  \h^0(M,\mathcal B\upi_1\Gamma) \ar[r] & 
 \h^0(M,\Gamma)\ar[r]& 
 \h^0(M,\idmorph {\upi_0\Gamma}) \ar`r[d]`[l]`[llld]`[dll][dll] \\
 & \h^1(M,\mathcal B\upi_1\Gamma) \ar[r] & 
 \h^1(M,\Gamma)\ar[r]& 
 \h^1(M,\idmorph {\upi_0\Gamma}) \ar`r[d]`[l]`[llld]`[dll][dll] \\ 
 &\h^2(M,\mathcal B\upi_1\Gamma)\text{.} & & 
}
\end{equation*}
\end{proposition}
The last set in the long exact sequence of Proposition \ref{long} is defined as the classical \v Cech cohomology group $\h^3_{cl}(M,\upi_1\Gamma)$ with values in the abelian Lie group 
$\upi_1\Gamma$. In contrast to this 
there is no way to define $ \h^2$ (or higher) for general Lie 2-groups $\Gamma$, so that the sequence stops there. 
Many variants of Proposition \ref{long} can be found in the literature. The case of the automorphism 2-group $\mathrm{AUT}(H)$ \cite[Example 2.4.4]{NW11} is treated in \cite{giraud}. Discrete crossed modules and algebraic variants are treated in \cite{dedecker1} and \cite{dedecker2}.

According to Theorem \ref{th:geo} the various cohomology sets occurring in the sequence of Proposition \ref{long} have natural geometric interpretations in terms of 
$\Gamma$-bundles or $\Gamma$-2-bundles. The first goal of this paper is to understand the exactness of the sequence geometrically in terms of reduction problems and lifting problems of these bundles. 
In particular, our results give an independent proof of Proposition \ref{long}.
Moreover, we will promote the (set-theoretical) exactness   statements to statements about (bi-)groupoids of (2-)bundles. 

We remark that distinguishing between \emph{lifting} and \emph{reduction} problems is purely a matter of taste: 
a lift is defined exactly in the same way as a reduction. We decided to use the terminology \emph{reduction} in situations where an object comes from a \quot{simpler} structure (e.g., a principal groupoid bundle \emph{reduces} to an ordinary abelian bundle, cf section \ref{sec:reduction}), 
while we use the terminology \emph{lift} when it comes from a \quot{more complicated} structure (e.g., a principal bundle \emph{lifts} to a non-abelian 2-bundle, see Section \ref{sec:lift2gerbe}).

\setsecnumdepth{1}

\section{Reduction and Lifting for Non-Abelian Gerbes}

\label{sec:liftings}
 
In this section we discuss reduction and lifting problems in the context of 2-bundles over a paracompact manifold $M$, and for a smoothly separable Lie 2-group $\Gamma$. 
This gives a geometric proof and understanding of the long exact sequence of Proposition \ref{long}.
%
%
Although we  formulate and prove the results of this section in the language of $\Gamma$-bundle gerbes, all results carry over to principal 2-bundles via the equivalence of Theorem \ref{thm:equiv}.  

\setsecnumdepth{2}

\subsection{Reduction to Abelian Gerbes}\label{sec:reduction}

Now we look at the sequence
\begin{equation}\label{sequence:bicat}
\xymatrix{
\grb{{\mathcal{B} \upi_1\Gamma}} M \ar[r]^-{i_*} & \grb\Gamma M \ar[r]^-{\pi_*} & \idmorph{\bun{{\mathcal{B}\upi_0\Gamma}} M}\text{.}
}
\end{equation}
of bigroupoids and 2-functors induced by the functors $i$ and $\pi$, 
where we implicitly used the canonical equivalence $\grb{{\idmorph{\upi_0\Gamma}}} M \cong \idmorph{\bun{{\mathcal{B}\upi_0\Gamma}} M}$ \cite[Example 5.1.9]{NW11}. 

We want to prove that this sequence is a ``fibre sequence'' in bigroupoids. To formulate this properly we note that the groupoid $\bun{{\mathcal{B}\upi_0\Gamma}} M$ has a canonical ``base point'', namely the trivial principal $\mathcal{B}\upi_0\Gamma$-bundle $\trivlin_1$. Next we give an explicit definition of the ``homotopy fibre'' of this base point.

\begin{definition}\phantomsection\label{def:oriented}
\begin{enumerate}[(i)]

\item Let $\mathcal{G}$ be a $\Gamma$-bundle gerbe over $M$. A \underline{$\pi$-orientation} of $\mathcal G$ is a global section of the principal bundle $\pi_*(\mathcal G)$. 

\item A 1-morphism $\varphi\maps \mathcal G \to \mathcal G'$ between $\pi$-oriented $\Gamma$-bundle gerbes is called \uline{$\pi$-orientation preserving} if the induced morphism $\pi_*(\varphi)$ of $\upi_0\Gamma$-bundles preserves the  sections, i.e $\pi_*(\varphi) \nobr\circ\nobr s \eq s'$.
\end{enumerate}
\end{definition}

The bigroupoid consisting of all $\pi$-oriented $\Gamma$-bundle gerbes, all $\pi$-orientation-preserving morphisms, and all 2-morphisms is denoted $\grbor{\Gamma}{M}\pi$.
The composition $\pi \circ i$ in sequence \erf{sequence:bicat} is  the trivial 2-group homomorphism. This implies directly the following lemma.

\begin{lemma}\label{lemma:teileins}
The 2-functor $\pi_* \circ i_*$ is canonically equivalent to the trivial 2-functor that sends each object to the trivial principal $\mathcal{B}\upi_0\Gamma$-bundle  and each morphism to the trivial morphism. In particular, the 2-functor $i_*$ lifts to a 2-functor
\begin{equation}
\label{eq:ior}
i_{or}\maps \grb {{\mathcal B \upi_1\Gamma}}{M} \to \grbor{\Gamma}{M}\pi\text{.}
\end{equation}
\end{lemma}

Let $\mathcal{G}$ be a $\Gamma$-bundle gerbe over $M$. A \emph{reduction of $\mathcal{G}$ to an abelian gerbe} is a $\mathcal{B}\upi_1\Gamma$-bundle gerbe $\mathcal{G}_{\mathrm{red}}$ such that $\mathcal{G} \cong i_{*}(\mathcal{G}_{\mathrm{red}})$. Lemma \ref{lemma:teileins}  shows one part of the ``exactness'' of sequence \erf{sequence:bicat}: reducible bundle gerbes are $\pi$-orientable. The other part is to show that the homotopy fibre $\grbor{\Gamma}{M}\pi$ agrees with $\grb{{B \upi_1\Gamma}} M$. 

\begin{theorem}
\label{reduction}
The 2-functor \erf{eq:ior} is an equivalence of bigroupoids:
\begin{equation*}
\grb {{\mathcal B \upi_1\Gamma}}{M} \cong \grbor{\Gamma}{M}\pi\text{.}
\end{equation*}
In particular, every $\pi$-orientation of a $\Gamma$-bundle gerbe $\mathcal{G}$ determines (up to isomorphism) a reduction of $\mathcal{G}$ to an abelian gerbe. \end{theorem}

\begin{proof}
We show first that $i_{or}$ is essentially surjective. Let $\calg$ be a $\Gamma$-bundle gerbe with  $\pi$-orientation $s$. We denote by $\zeta\maps Y \to M$ the surjective submersion of $\mathcal{G}$, by $P$ its principal $\Gamma$-bundle over $Y^{[2]}$, and by $\mu$ its bundle gerbe product. The $\pi$-orientation $s$ can be described locally as a map $s\maps Y \to \upi_0\Gamma$ satisfying the condition
\begin{equation}\label{schnitt}
\pi_*(P) \cdot \zeta_1^*s \eq \zeta_2^*s
\end{equation}
where $\pi_*(P) \maps Y^{[2]} \to \upi_0\Gamma$. We may assume that $s$ lifts along $\pi\maps \Gamma_0 \to \upi_0\Gamma$ to a smooth map $t\maps Y \to \Gamma_0$, otherwise we pass to an isomorphic bundle gerbe by a refinement of the surjective submersion $\zeta$.  Now we consider the following $\mathcal{B}\upi_1\Gamma$-bundle gerbe $\calg_{\mathrm{red}}$. Its surjective submersion is $\zeta$. We recall from \cite{NW11}[Example 2.2.3] that for each smooth map $f:X \to \upi_0\Gamma$ there is a \emph{trivial principal $\Gamma$-bundle} $\trivlin_f$ over $X$.
The principal $\Gamma$-bundle
\begin{equation*}
P' \df \trivlin_{(\zeta_2^*t)^{-1}} \otimes P \otimes \trivlin_{\zeta_1^*t} \end{equation*}
satisfies $\pi_{*}(P')\eq1$ due to \erf{schnitt}, and so defines a principal $\mathcal{B}\upi_1\Gamma$-bundle $P'_{\mathrm{red}}$; this is the principal bundle of $\mathcal{G}_{\mathrm{red}}$. 
Finally,  the isomorphism 
\begin{equation*}
\mu'\df \id \otimes \mu \otimes \id\maps \quad \trivlin_{(\zeta_3^*t)^{-1}} \otimes \zeta_{23}^*P \otimes \zeta_{12}^*P \otimes \trivlin_{\zeta_1^*t}
\to \trivlin_{(\zeta_3^*t)^{-1}} \otimes \zeta_{13}^*P \otimes \trivlin_{\zeta_1^*t}
\end{equation*}
of principal $\Gamma$-bundles over $Y^{[3]}$ reduces to a bundle gerbe product for $\mathcal{G}_{\mathrm{red}}$. We claim that $\calg$ and $i_{*}(\calg_{\mathrm{red}})$ are isomorphic in $\grb \Gamma M^{\pi-or}$. Indeed, an orientation preserving isomorphism $\calg \to i_{*}(\calg_{\mathrm{red}})$ is given by the pair $(\trivlin_{t^{-1}},\id)$.

It remains to show that $i_{or}$ is fully faithful. We consider two $\mathcal{B}\upi_1\Gamma$-bundle gerbes $\mathcal{G}_1$ and $\mathcal{G}_2$. 
Suppose $\mathcal{A}\maps i_{*}(\mathcal{G}_1) \to i_{*}(\mathcal{G}_2)$  is a 1-morphism that respects 
the canonical $\pi$-orientations of Lemma \ref{lemma:teileins}. If $\mathcal{A}$ consists of a principal $\Gamma$-bundle over some 
common refinement $Z$, it follows that $\pi_{*}(Q)$ descends to a smooth map $q\maps M \to \upi_0\Gamma$. 
The condition that $\mathcal{A}$ preserves the $\pi$-orientations requires $q$ to be the constant map. 
This in turn shows that $\pi_{*}Q\eq1$, which implies that $\mathcal{A} \cong i_{*}(\mathcal{A}_{\mathrm{red}})$ by 
 similar reduction as above. This shows that $i_{or}$ is essentially surjective on Hom-groupoids. That it is fully faithful 
on Hom-groupoids follows by observing 
that the 2-morphisms between 1-morphisms $i_{*}(\mathcal{B}_1)$ and $i_{*}(\mathcal{B}_2)$ are 
exactly the 2-morphisms between $\mathcal{B}_1$ and $\mathcal{B}_2$.
\end{proof}

\begin{corollary}
The sequence of Proposition \erf{long} is exact at $\h^1(M,\Gamma)$.  
\end{corollary}

\begin{example}
The terminology \quot{$\pi$-orientation} is  inspired by the  example of \textit{Jandl gerbes}, which play an important role in unoriented sigma models \cite{schreiber1,nikolaus2}. The Jandl-2-group $\mathcal J U(1)$ is the 2-group induced by the 
crossed module 
\begin{equation*}
\alxydim{}{U(1) \ar[r]^{0} & \mathbb{Z} / 2}
\end{equation*}
with $\mathbb{Z} /2 $ acting on $U(1)$ by inversion; see \cite[Section 2.4]{NW11}. The fibre sequence \eqref{sequence} is here
$$ \mathcal B U(1) \to \mathcal J U(1) \to \idmorph{(\mathbb{Z}/2)}\text{.}$$
A $\mathcal J U(1)$-bundle gerbe $\mathcal G$ is called 
$\textit{Jandl gerbe}$ 
 \cite{nikolaus2} (there is also a connection whose discussion  we omit here). Now,
 $\pi_*(\mathcal G)$ is a $\mathbb{Z}/2$-bundle called the 
\textit{orientation bundle} of the Jandl gerbe.  It is crucial for the definition of unoriented surface holonomy. Hence Theorem \ref{reduction} shows that an oriented Jandl gerbe is just an ordinary  $\mathcal B U(1)$-bundle gerbe.
\end{example}

\subsection{The lifting bundle 2-gerbe}\label{sec:lift2gerbe}

In this section we come to the last step of the exact sequence of Proposition \ref{long}, namely to the connecting homomorphism 
\begin{equation*}
\h^1(M, \idmorph{\upi_0\Gamma}) \to \h^2(M, \mathcal{B} \upi_1 \Gamma)\text{.}
\end{equation*}
The geometric objects that represent classes in $\h^2(M, \mathcal{B} \upi_1\Gamma)$ are \emph{$\upi_1\Gamma$-bundle 2-gerbes}  \cite{stevenson2}. A bundle 2-gerbe is a higher analogue of a bundle gerbe, and can be defined for a general abelian Lie-group $A$. 

\begin{definition}
An \underline{$A$-bundle 2-gerbe} over $M$ is a surjective submersion $\pi\maps Y \to M$, a $\mathcal{B}A$-bundle gerbe $\calg$ over $Y^{[2]}$, a 1-isomorphism 
$$ \mathcal{M}\maps \pi_{23}^*\calg \otimes \pi_{12}^* \calg \to \pi_{13}^* \calg$$
of $\mathcal{B}A$-bundle gerbes over $Y^{[3]}$ and a 2-isomorphism 
\begin{equation*}
\alxydim{@R=1.2cm}{
\pi_{34}^*\mathcal{G} \otimes \pi_{23}^*\calg \otimes \pi_{12}^* \calg \ar[r]\ar[d]& \pi_{34}^*\calg \otimes \pi_{13}^* \calg\ar[d]\ar@{=>}[ld]|*+{\mu} \\
\pi_{24}^*\mathcal{G} \otimes \pi_{12}^*\calg  \ar[r]                        & \pi_{14}^* \calg 
}
\end{equation*}
over $Y^{[4]}$ satisfying the pentagon axiom.
\end{definition}
\begin{remark}
\begin{enumerate}
\item
Following the terminology of Definition \ref{def:grb} and the cohomological count it would be more logical to call this a $\mathcal{BB}A$-bundle 2-gerbe. But we have decided to follow the naming used in the literature, also because bundle 2-gerbes can not be defined for general Lie 2-groups $\Gamma$.

\item 
As pointed out above, every $A$-bundle 2-gerbe $\mathbb{G}$ has a characteristic class $[\mathbb{G}] \in \h^2(M,\mathcal{B}A)$. Bundle 2-gerbes are up to isomorphism classified by this class \cite{stevenson2}.
\end{enumerate}
\end{remark}

We note that the Lie 2-group $\Gamma$ determines a $\mathcal{B}\upi_1\Gamma$-bundle gerbe $\mathcal{G}_{\Gamma}$ over $\upi_0\Gamma$~: its submersion is the projection $\pi\maps \Gamma_0 \to \upi_0\Gamma$, its principal $\mathcal{B}\upi_1\Gamma$-bundle over $\Gamma_0^{[2]}$ is  $\Gamma_1$, and the multiplication $\mu$ is given by the composition in $\Gamma$. We infer that $\mathcal{G}_{\Gamma}$ is multiplicative in the sense of \cite{carey4}, which means that is carries the following structure.
\begin{enumerate}[(i)]
\item
There is an isomorphism
$$ \mathcal{M}_\Gamma\maps p_1^* \calg_\Gamma \otimes p_2^* \calg_\Gamma \to m^*\calg_\Gamma$$
over $\upi_0\Gamma \times \upi_0\Gamma$, where $p_1,p_2$ are the projections and $m$ is the multiplication.

In order to construct $\mathcal{M}_{\Gamma}$ we recall that $\mathcal{G}_{\Gamma}$ is the bundle gerbe associated to the principal $\mathcal{B}\upi_1\Gamma$-2-bundle $\Gamma$ over $\upi_0\Gamma$ via the 2-functor $\mathscr{E}$ of \cite[Section 7.1]{NW11}. Since $\upi_1\Gamma$ is central in $\Gamma_1$, the 2-group multiplication $M\maps \Gamma \times \Gamma \to \Gamma$ can be seen as a 2-bundle morphism
\begin{equation*}
M\maps P _1^{*}\Gamma \otimes p_2^{*}\Gamma \to m^{*}\Gamma
\end{equation*}
over $\upi_0\Gamma \times \upi_0\Gamma$. Since the 2-functor $\mathscr{E}$ is moreover a morphism between pre-2-stacks \cite[Proposition 7.1.8]{NW11}, it converts this 2-bundle morphism into the claimed isomorphism $\mathcal{M}_{\Gamma}$.

\item
The isomorphism $\mathcal{M}_{\Gamma}$ is associative in the sense that there is a 2-isomorphism 
\begin{equation*}
\alxydim{@R=1.2cm}{
p_{34}^*\mathcal{G} \otimes p_{23}^*\calg \otimes p_{12}^* \calg \ar[r]\ar[d]& p_{34}^*\calg \otimes p_{13}^* \calg\ar[d]\ar@{=>}[ld]|*+{\mu_\Gamma~}\\
p_{24}^* \otimes p_{12}^*\calg  \ar[r]                        & p_{14}^* \calg 
}
\end{equation*}
that satisfies the pentagon identity. This 2-isomorphism is simply induced by the fact that the 2-group multiplicative $M$ is strictly associative.

%
\end{enumerate}

Now let $E$ be a principal $\upi_0\Gamma$-bundle over $M$. The idea is to define a $\upi_1\Gamma$-bundle 2-gerbe $\mathbb{L}_E$ whose surjective submersion is the bundle projection $E \to M$. We denote by $\delta_n\maps  E^{[n+1]} \to \upi_0\Gamma^{n}$ the  \quot{difference maps} given by $e_0 \cdot \mathrm{pr}_i(\delta_{n}(e_0,e_1,\ldots,e_n)) \eq e_i$.

\begin{definition}
\label{def:lifting2gerbe}
Let $E$ be a principal  $\upi_0\Gamma$-bundle. The \emph{lifting bundle 2-gerbe} $\mathbb{L}_E$ is given by the surjective submersion 
$E \to M$, the bundle gerbe $\delta_1^*\calg_\Gamma$, the multiplication $\delta_2^*\mathcal{M}_\Gamma$ and the associator $\delta_3^*\mu_\Gamma$.
\end{definition}

\begin{remark}
\label{rem:cs}
In the context of Chern-Simons theories with a gauge group $G$, there exists a bundle 2-gerbe $\mathbb{CS}_E(\mathcal{G})$ constructed from a principal $G$-bundle $E$ over $M$ and a multiplicative $\mathcal{B}\ueins$-bundle gerbe $\mathcal{G}$ over $G$ \cite{carey4,waldorf5}. A similar construction has also been proposed in \cite{jurco1}. In the particular case that $G\eq \upi_0\Gamma$ and $\ueins \eq \upi_1\Gamma$,  these constructions coincide with the one of Definition \ref{def:lifting2gerbe}, i.e.,
\begin{equation*}
\mathbb{CS}_E(\mathcal{G}_{\Gamma}) \eq \mathbb{L}_E\text{.}
\end{equation*}
\end{remark}

Next we justify the name \quot{lifting bundle 2-gerbe}: we show that $\mathbb{L}_E$ is  the obstruction to lift the bundle $E$ to a $\Gamma$-bundle gerbe. Moreover, the possible lifts of $E$  correspond to trivializations of $\mathbb{L}_E$. In order to produce a precise statement, we first give an explicit definition of the homotopy fibre of the 2-functor
\begin{equation*}
\pi_* \maps  \grb\Gamma M \to \idmorph{\bun{{\mathcal{B}\upi_0\Gamma}} M}\text{.}
\end{equation*}

\begin{definition}\label{def:lift}
Let $E$ be a     principal $\upi_0\Gamma$-bundle. The bigroupoid $\lift_\Gamma(E)$ is defined as follows.
\begin{itemize}

\item An object is a \emph{$\Gamma$-lift of $E$}: a  pair $(\mathcal{G},\varphi)$ consisting of a  $\Gamma$-bundle gerbe $\calg$ and of a  bundle isomorphism $\varphi\maps \pi_*(\calg) \to E$.

\item 
A 1-morphism between objects $(\mathcal{G}_1,\varphi_1)$ and $(\mathcal{G}_2,\varphi_2)$ is a 1-morphism $\mathcal{B}\maps  \mathcal{G}_1 \to \mathcal{G}_2$ such that $\varphi_2 \circ \pi_{*}(\mathcal{B}) \eq \varphi_1$.

\item 
A 2-morphism between 1-morphisms $\mathcal{B}$ and $\mathcal{B}'$ is just a 2-morphism. 
\end{itemize}
\end{definition}

\begin{remark}
For the trivial $\upi_0\Gamma$-bundle $E \eq M \times \upi_0\Gamma$ the groupoid $\lift_\Gamma(E)$ agrees with the groupoid of $\pi$-oriented $\Gamma$-bundle gerbes introduced in Definition \ref{def:oriented}.
\end{remark}

The bigroupoid we want to compare with $\lift_\Gamma(E)$ is the 2-groupoid $\triv(\mathbb{L}_E)$ of trivializations of the lifting gerbe $\mathbb{L}_E$; see  {\cite[Definition 11.1]{stevenson2}} and \cite[Section 5.1]{waldorf8}. 

\begin{theorem}\label{thm:lift}
Let $E$ be a principal  $\upi_0\Gamma$-bundle over $M$, and let $\mathbb{L}_E$ be the associated lifting bundle 2-gerbe $\mathbb{L}_E$. Then there is an equivalence of bigroupoids
$$ \triv(\mathbb{L}_E) \cong \lift_\Gamma(E).$$
In particular, $E$ admits a $\Gamma$-lift  if and only if $\mathbb{L}_E$ is trivializable.
\end{theorem}

Theorem \ref{thm:lift} is proved in Section \ref{sec:proof}. Before that we want to present two corollaries. First we recall that the bigroupoid of trivializations of an $A$-bundle 2-gerbe is a torsor over the monoidal bigroupoid of $\mathcal{B}A$-bundle gerbes \cite[Lemma 2.2.5]{waldorf8}. From Theorem \ref{thm:lift} we get the following two implications:

\begin{corollary}
\label{co:action}
The bigroupoid $\lift_\Gamma(E)$ is a torsor over the monoidal bigroupoid of $\mathcal{B}\upi_1\Gamma$-bundle gerbes over $M$, i.e., the $\mathcal{B}\upi_1\Gamma$-bundle gerbes over $M$ act on the $\Gamma$-lifts of $E$ in such a way that on isomorphism classes of objects a free and transitive action is induced.
\end{corollary}

\begin{corollary}
The sequence of Proposition \ref{long} is exact at $\h^1(M,\idmorph {\upi_0\Gamma})$. \end{corollary}

\subsection{Proof of Theorem \ref{thm:lift}}

\label{sec:proof}

Our strategy to prove Theorem \ref{thm:lift} is to reduce it to Theorem \ref{reduction} using descent theory. First we need the following preliminaries. 

\begin{enumerate}[(i)]
\item  
Let $\calg$ and $\calg'$  be $\pi$-oriented $\Gamma$-bundle gerbes over
 a manifold $X$. For a 1-morphism  $\mathcal{A}\maps  \calg \to \calg'$ we obtain a smooth map $h_\mathcal{A}\maps  X \to \upi_0\Gamma$  
determined by 
\begin{equation*}
\pi_*(\mathcal{A})(s(x)) \cdot h_{\mathcal{A}}(x) \eq s'(x)\text{,}
\end{equation*}
where $s$ and $s'$ are the $\pi$-orientations of $\mathcal{G}$ and $\mathcal{G}'$, respectively.
We have $h_{\mathcal{A}}\equiv 1$ if and only if $\mathcal{A}$ is $\pi$-oriented.

\item
For a  smooth map $h\maps  X \to \upi_0\Gamma$ we denote
 by  $\mathrm{Hom}^h(\calg,\calg')$ the full subgroupoid of the Hom-groupoid $\mathrm{Hom}_{\grb \Gamma X}(\calg,\calg')$ over those 1-morphisms $\mathcal{A}$ with $h_\mathcal{A} \eq h$.
\end{enumerate}

\begin{lemma}
\label{lem:trivggamma}
The $\Gamma$-bundle gerbe $i_{*}(\mathcal{G}_{\Gamma})$ is trivializable.
\end{lemma}

\begin{proof}
A trivialization is given by the principal $\Gamma$-bundle $T \df \trivlin_{i}$ over $\Gamma_0$, for $i\maps\Gamma_0 \to \Gamma_0$ the inversion of the group $\Gamma_0$, and by the bundle isomorphism
\begin{equation*}
\alxydim{@C=1.5cm}{\trivlin_1 \otimes \pi_1^{*}T \cong \trivlin_{\mathrm{pr}_2} \otimes \trivlin_{\Delta}  \ar[r]^-{\id \otimes \tau^{-1}} &   \pi_2^{*}T \otimes i_{*}(\Gamma_1)}
\end{equation*}
over $\Gamma_0^{[2]}$ which is defined using the obvious bundle isomorphism $\tau \maps i_{*}(\Gamma_1) \to \trivlin_{\Delta}$. The required compatibility with the bundle gerbe product is easy to check. \end{proof}

\begin{lemma}\label{hilfslemma}
Let $h\maps  X \to \upi_0\Gamma$ be a smooth map and $\calg$ and $\calg'$ be $\mathcal{B}\upi_1\Gamma$-bundle gerbes over $X$.
\begin{enumerate}[(i)]
\item 
We have an equivalence of groupoids 
\begin{equation*}
t\maps \mathrm{Hom}(\calg \otimes h^*\calg_\Gamma,\calg') \to \mathrm{Hom}^h(i_{or}(\calg),i_{or}(\calg'))\text{.}
\end{equation*}

\item
Let $h'\maps  X \to \upi_0\Gamma$ be another smooth map. The diagram
\begin{equation*}
\alxydim{@R=1.2cm}{\mathrm{Hom}(\mathcal{G}' \otimes h'^{*}\mathcal{G}_{\Gamma},\mathcal{G}'') \times \mathrm{Hom}(\mathcal{G} \otimes h^{*}\mathcal{G}_{\Gamma},\mathcal{G}') \ar[d]_{t \times t}\ar[r] & \mathrm{Hom}(\mathcal{G} \otimes (hh')^{*}\mathcal{G}_{\Gamma},\mathcal{G}'')\ar[d]^{t} \\ \mathrm{Hom}^{h'}(i_{or}(\mathcal{G}'),i_{or}(\mathcal{G}'')) \times \mathrm{Hom}^{h}(i_{or}(\mathcal{G}),i_{or}(\mathcal{G}'))  \ar[r]_-{\circ} & \mathrm{Hom}^{hh'}(i_{or}(\mathcal{G}),i_{or}(\mathcal{G}''))  } \end{equation*}
of functors, in which the  arrow in the first row is given by
\begin{equation*}
(\mathcal{B},\mathcal{A}) \mapsto \mathcal{B} \circ (\mathcal{A} \otimes \id) \circ (\id_{\mathcal{G}} \otimes (h \times h')^{*}\mathcal{M}_{\Gamma}^{-1})\text{,} 
\end{equation*}
is commutative up to an associative natural equivalence.

\end{enumerate}
\end{lemma}

\begin{proof}
If $\mathcal{G}$ is a $\pi$-oriented $\Gamma$-bundle gerbe  over $X$ and $h\maps  X \to \upi_0\Gamma$ is a smooth map, we obtain another $\pi$-oriented $\Gamma$-bundle gerbe denoted $\mathcal{G}_h$, which is the same $\Gamma$-bundle gerbe equipped with  the new $\pi$-orientation  $s_h \df s \cdot h$, where $s$ is the original $\pi$-orientation. We note that
\begin{equation}
\label{eq:htrans}
\mathrm{Hom}^h(\mathcal{G},\mathcal{G}') \eq \mathrm{Hom}^1(\mathcal{G}_h,\mathcal{G}')\text{.}
\end{equation}
In order to prove (i) we shall construct an orientation-preserving 1-isomorphism 
\begin{equation*}
\mathcal{C}_h\maps  i_{or}(\calg \otimes h^*\calg_\Gamma) \to i_{or}(\calg)_h\text{,}
\end{equation*}
which is by \erf{eq:htrans} the same as an 1-isomorphism of $\Gamma$-bundle gerbes with $h_{\mathcal{C}_h}\eq h$. It is indeed clear that the trivialization $\mathcal{T}_{\Gamma}$ of $i_{*}(\mathcal{G}_{\Gamma})$ 
constructed in Lemma \ref{lem:trivggamma} induces a 1-isomorphism of  $\Gamma$-bundle gerbes, and it is straightforward to check that this isomorphism satisfies the condition $h_{\mathcal{C}_h}\eq h$. Then, the equivalence $t$ is the composition of the following equivalences:
\begin{equation*}
\alxydim{@R=1.2cm}{\mathrm{Hom}(\calg \otimes h^*\calg_\Gamma,\calg') \ar[r]^-{i_{or}} & \mathrm{Hom}^1(i_{or}(\calg \otimes h^*\calg_\Gamma),i_{or}(\calg')) \ar[d]^-{\mathrm{Hom}(\mathcal{C}_h,-)} \\ & \mathrm{Hom}^1(i_{or}(\calg )_h,i_{or}(\calg')) \ar@{=}[r]^{\erf{eq:htrans}} & \mathrm{Hom}^h(i_{or}(\mathcal{G}),i_{or}(\mathcal{G}'))\text{.}}
\end{equation*}
The construction of the natural equivalence in (ii) is now straightforward.
\end{proof}

Now let $E$ be a principal $\mathcal{B}\upi_0\Gamma$-bundle over $M$ and let $\mathbb{L}_E$ be the associated lifting bundle 2-gerbe. We first want to give another description of the groupoid $\lift_\Gamma(E)$. Since the projection $\zeta\maps  E \to M$ is a surjective submersion and $\Gamma$-bundle gerbes form a stack  \cite[Theorem 3.3]{nikolaus2}, we have an equivalence 
\begin{equation*}
\grb \Gamma M \cong \Desc_\zeta(\grb \Gamma -)
\end{equation*}
between the bigroupoid $\Gamma$-bundle gerbes over $M$, and the bigroupoid of descent data with respect to $\zeta\maps  E \to M$; for the notation we refer the reader to Section 2 of \cite{nikolaus2}. An object in $\Desc_\zeta(\grb \Gamma -) $ is:
\begin{enumerate}
\item a $\Gamma$-bundle gerbe $\calg$ over $E$,
\item a 1-isomorphism $\mathcal{A}\maps  \mathrm{pr}_2^*\calg \to \mathrm{pr}_1^*\calg$ over $E^{[2]},$

\item a 2-isomorphism $\chi\maps \mathrm{pr}_{23}^*\mathcal{A} \circ \mathrm{pr}_{12}^{*}\mathcal{A} \Rightarrow \mathrm{pr}^*_{13}\mathcal{A}$ over $E^{[3]}$, and  
\item a coherence condition for $\chi$ over $E^{[4]}$.
\end{enumerate}
Analogously, there is an equivalence
\begin{equation*}
\bun{\mathcal{B}\upi_0\Gamma}{M} \cong \Desc_\zeta(\bun{\mathcal{B}\upi_0\Gamma}{-})
\end{equation*}
between the groupoid of principal $\mathcal{B}\upi_0\Gamma$-bundles and their descent data. 
Explicitly, the principal $\mathcal{B}\upi_0\Gamma$-bundle $E$ over $M$ corresponds to the following descent object:
\begin{enumerate}
\item the trivial $\mathcal{B}\upi_0\Gamma$-bundle $\trivlin$ over $E$,
\item the isomorphism $\mathrm{pr}_2^{*}\trivlin \to \mathrm{pr}_1^{*}\trivlin$  induced by the difference map $\delta_2\maps  E \times_M E \to \upi_0\Gamma$, and
\item a cocycle condition over $E^{[3]}$. 
\end{enumerate}

Now we are in a position to give a descent-theoretical formulation of the bigroupoid $\lift_{\Gamma} (E)$. We obtain a bigroupoid $\mathcal{D}$, in which an object consists of:
\begin{enumerate}[(a)]

\item a $\pi$-oriented  $\Gamma$-bundle gerbe $\calg$ over $E$, 

\item a 1-isomorphism $\mathcal{A}\maps  \mathrm{pr}_2^*\calg \to \mathrm{pr}_1^*\calg$ of $\Gamma$-bundle gerbes over $E^{[2]}$ such that $h_{\mathcal{A}} \eq \delta$, i.e., an object in $\mathrm{Hom}^\delta(\mathrm{pr}_2^*\calg,\mathrm{pr}_1^*\calg)$, 
\item a 2-isomorphism $\mu\maps  \mathrm{pr}_{23}^*\mathcal{A} \circ \mathrm{pr}_{12}^*\mathcal{A} \Rightarrow \mathrm{pr}^*_{13}\mathcal{A}$ over $E^{[3]}$, i.e., a morphism 
in $\mathrm{Hom}^{\delta'}\big(\mathrm{pr}_3^*\calg,\mathrm{pr}_1^*\calg\big)$ where 
$\delta'$  denotes the map $\mathrm{pr}_{23}^*\delta \cdot \mathrm{pr}_{12}^*\delta \eq \mathrm{pr}_{13}^*\delta$, and 

\item a coherence condition for $\mu$ over $E^{[4]}$.
\end{enumerate}
The 1-morphisms and 2-morphisms in the bigroupoid $\mathcal{D}$ are defined in the same evident way. 
By  \cite[Theorem 3.3]{nikolaus2} the two bigroupoids are equivalent: $\lift_{\Gamma} (E) \cong \mathcal{D}$. Functors in both directions are given by pullback and descent along $\zeta\maps E \to M$.

The next step is to translate the structure of the bigroupoid $\mathcal{D}$ using the equivalences of Theorem \ref{reduction}  and Lemma \ref{hilfslemma}. We  obtain yet another bigroupoid $\mathcal{E}$ whose objects consist of
\begin{enumerate}[(a')]

\item 
a $\mathcal{B}\upi_1\Gamma$-gerbe $\mathcal{H}$ over $E$, which corresponds to the $\pi$-oriented $\Gamma$-bundle gerbe $\mathcal{G}$ of (a) under the equivalence of Theorem \ref{reduction},

\item
a 1-isomorphism $\mathcal{B}\maps  \mathrm{pr}_2^{*}\mathcal{H} \otimes \delta^{*}\mathcal{G}_{\Gamma} \to \mathrm{pr}_1^{*}\mathcal{H}$ over $E^{[2]}$, which corresponds to the 1-isomorphism $\mathcal{A}$ of (b) under the equivalence of Lemma \ref{hilfslemma} (i),

\item
a 2-isomorphism 
\begin{equation*}
\alxydim{@R=1.2cm@C=2cm}{\mathrm{pr}_3^{*}\mathcal{H} \otimes \delta_{23}^{*}\mathcal{G}_{\Gamma} \otimes \delta_{12}^{*}\mathcal{G}_{\Gamma}  \ar[r]^-{\mathrm{pr}_{23}^{*}\mathcal{B} \otimes \id} \ar[d]_{\id \otimes \mathcal{M}_{\Gamma}} & \mathrm{pr}_2^{*}\mathcal{H} \otimes \delta_{12}^{*}\mathcal{G}_{\Gamma} \ar@{=>}[dl]|*+{\alpha} \ar[d]^{\mathrm{pr}_{12}^{*}\mathcal{B}} \\ \mathrm{pr}_3^{*}\mathcal{H} \otimes \delta_{13}^{*}\mathcal{G}_{\Gamma} \ar[r]_{\mathrm{pr}_{13}^{*}\mathcal{B}} & \mathrm{pr}_1^{*}\mathcal{H}}
\end{equation*} 
over $E^{[3]}$, which corresponds essentially to the 2-isomorphism $\mu$ of (c) under the equivalence of Lemma \ref{hilfslemma} (i), with an additional whiskering  using the diagram of Lemma \ref{hilfslemma} (ii), and 

\item
a coherence condition for $\alpha$  over $E^{[4]}$.

\end{enumerate}
Since the translations we applied are \emph{2-functorial} (Theorem \ref{reduction} and Lemma \ref{hilfslemma}), a similar description can easily be given for 1- and 2-morphisms of the bigroupoid $\mathcal{E}$. Since they are \emph{equivalences},  the bigroupoids $\mathcal{D}$ and $\mathcal{E}$ are equivalent. 

In order to finish the proof of Theorem \ref{thm:lift} it remains to notice that the bigroupoid $\mathcal{E}$ \emph{is} the bigroupoid $\triv(\mathbb{L}_E)$ of trivializations of $\mathbb{L}_{E}$, just as defined in  \cite[Section 5.1]{waldorf8}.

\setsecnumdepth{2}

\section{Transgression for Non-Abelian Gerbes}

\label{sec:transgression}

In this section, $M$ is a smooth, finite-dimensional manifold. Transgression requires to work with connections on bundle gerbes, and we refer the reader to \cite{waldorf5,waldorf10} and references therein for a complete discussion using the same notation we are going to use here. 
Generally, if $A$ is an abelian Lie group, and $\mathcal{G}$ is a $\mathcal{B}A$-bundle gerbe with connection, there is a  monoidal 2-functor
\begin{equation*}
\mathscr{T}\maps  \grbcon {\mathcal{B}A} M \to \idmorph{\bun {\mathcal{B}A}{LM}}
\end{equation*}
called \emph{transgression}, which takes $\mathcal{B}A$-bundle gerbes with connection over $M$ to principal $\mathcal{B}A$-bundles  over the free loop space $LM \df C^{\infty}(S^1,M)$. We remark that the bundles in the image of $\mathscr{T}$ are actually equipped with more structure \cite{waldorf10}, which we do not need here. On the level of isomorphism classes, the dependence on the connections drops out, and the 2-functor $\mathscr{T}$ induces a group homomorphism
\begin{equation*}
\tau\maps  \h^1(M,\mathcal{B}A) \to \h^0(LM,\mathcal{B}A)\text{.}
\end{equation*}
The goal of the present section is to generalize this transgression homomorphism from $\mathcal{B}A$ to a general, smoothly separable  Lie 2-group with $\upi_0\Gamma$  compact and connected.

\subsection{The Loop Group of a Lie 2-Group}

We consider a smoothly separable Lie 2-group $\Gamma$ with $\upi_0\Gamma$  compact. We denote by $\mathcal{G}_{\Gamma}$ the associated   $\mathcal{B}\upi_1\Gamma$-bundle gerbe over $\upi_0\Gamma$; see Section \ref{sec:lift2gerbe}.
 In Section \ref{sec:lift2gerbe} we have equipped $\mathcal{G}_{\Gamma}$ with a multiplicative structure, consisting of a 1-isomorphism $\mathcal{M}_{\Gamma}$ over $G \times G$ and of a 2-isomorphism $\alpha$ over $G^3$. Since $\upi_0\Gamma$ is compact, $\mathcal{G}_{\Gamma}$ admits a  \emph{multiplicative connection}  \cite[Proposition 2.3.8]{waldorf5}.

We recall in more generality from \cite[Definition 1.3]{waldorf5} that a multiplicative  connection on a multiplicative $\mathcal{B}A$-bundle gerbe $(\mathcal{G},\mathcal{M},\alpha)$ over a  Lie group $G$ is a connection $\lambda$ on $\mathcal{G}$, (together denoted by $\mathcal{G}^{\lambda}$), a 2-form $\rho \in \Omega^2(G^2,\mathfrak{a})$ with values in the Lie algebra $\mathfrak{a}$ of $A$, and a connection $\eta$ on $\mathcal{M}$ (together denoted by $\mathcal{M}^{\eta}$), such that
\begin{equation}
\label{eq:multstr}
\mathcal{M}^{\eta} \maps  \mathrm{pr}_1^{*}\mathcal{G}^{\lambda} \otimes \mathrm{pr}_2^{*}\mathcal{G}^{\lambda} \to m^{*}\mathcal{G}^{\lambda} \otimes \mathcal{I}_{\rho}
\end{equation}
is a connection-preserving 1-isomorphism, and $\alpha$ is a connection-preserving 2-isomorphism. In \erf{eq:multstr}  we have denoted by $\mathcal{I}_{\rho}$ the trivial bundle gerbe equipped with the curving $\rho$. The 2-form $\rho$  is in fact determined by $\lambda$ and $\eta$ and the assumption that \erf{eq:multstr} is connection-preserving. 

\noindent
Applying the transgression 2-functor $\mathscr{T}$, we obtain:
\begin{enumerate}[(a)]

\item 
a principal $\mathcal{B}A$-bundle $\mathscr{T}_{\mathcal{G}^{\lambda}}$ over $LG$,

\item
a bundle isomorphism
\begin{equation*}
\alxydim{@C=1.5cm}{\mathrm{pr}_1^{*}\mathscr{T}_{\mathcal{G}^{\lambda}} \otimes \mathrm{pr}_2^{*}\mathscr{T}_{\mathcal{G}^{\lambda}} \ar[r]^-{\mathscr{T}_{\mathcal{M}^{\eta}}} & m^{*}\mathscr{T}_{\mathcal{G}^{\lambda}} \otimes \mathscr{T}_{\mathcal{I}_{\rho}} \ar[r]^-{\id \otimes t_{\rho}} & m^{*}\mathscr{T}_{\mathcal{G}^{\lambda}} }
\end{equation*}
over $LG^2$, using the fact that the transgression of the trivial bundle gerbe $\mathcal{I}_{\rho}$ has a canonical trivialization $t_{\rho}$,

\item
an associativity condition for the bundle isomorphism of (b) over $LG^3$, coming from the transgression of the 2-isomorphism $\alpha$.

\end{enumerate}
We recall the following result.

\begin{theorem}[{{\cite[Theorem 3.1.7]{waldorf5}}}]
\label{th:central}
If $G$ is finite-dimensional, the bundle isomorphism of (b) equips the total space of $\mathscr{T}_{\mathcal{G}^{\lambda}}$ with the structure of a Fréchet Lie group, which we denote by $L\mathcal{G}^{\lambda,\eta}$. Moreover, $L\mathcal{G}^{\lambda,\eta}$ is a central extension
\begin{equation}
\label{eq:ex}
1 \to A \to L\mathcal{G}^{\lambda,\eta} \to LG \to 1
\end{equation}
of Fréchet Lie groups.  
\end{theorem}

We apply Theorem \ref{th:central} to the multiplicative $\mathcal{B}\upi_1{\Gamma}$-bundle gerbe $\mathcal{G}_{\Gamma}$, and some multiplicative connection $(\lambda,\eta)$ on it (note that $G\eq \upi_0\Gamma$ is by assumption compact, in particular finite-dimensional). Since we understand the connection as an auxiliary structure, we have to control different choices. For this purpose we need the following technical lemma. 

\begin{lemma}
\label{lem:1}
Let $(\mathcal{G},\mathcal{M},\alpha)$ be a multiplicative $A$-bundle gerbe over $G$, and let $(\lambda,\eta)$ and $(\lambda',\eta')$ be  multiplicative connections, determining 2-forms $\rho,\rho'\in\Omega^2(G \times G,\mathfrak{a})$. Then, there is a  2-form $\beta\in \Omega^2(G,\mathfrak{a})$ satisfying $\rho \eq \rho' + \Delta\beta$, and a connection $\varepsilon$ on the identity 1-isomorphism $\id\maps  \mathcal{G} \to \mathcal{G}$ such that
\begin{equation*}
\id^{\varepsilon}\maps  \mathcal{G}^{\lambda} \to \mathcal{G}^{\lambda'} \otimes \mathcal{I}_{\beta}
\end{equation*}
is connection-preserving, and the canonical 2-isomorphism
\begin{equation*}
\alxydim{@C=1.5cm@R=1.2cm}{\mathrm{pr}_1^{*}\mathcal{G}^{\lambda} \otimes \mathrm{pr}_2^{*}\mathcal{G^{\lambda}} \ar[r]^-{\mathcal{M}^{\eta}} \ar[d]_{\mathrm{pr}_1^{*}\id_{\varepsilon} \otimes \mathrm{pr}_2^{*}\id_{\varepsilon}} & m^{*}\mathcal{G}^{\lambda} \otimes \mathcal{I}_{\rho} \ar@{=>}[ddl] \ar[d]^{m^{*}\id_{\varepsilon} \otimes \id_0} \\ \mathrm{pr}_1^{*}(\mathcal{G}^{\lambda'} \otimes \mathcal{I}_{\beta}) \otimes \mathrm{pr}_2^{*}(\mathcal{G}^{\lambda'} \otimes \mathcal{I}_{\beta}) \ar@{=}[d] & m^{*}(\mathcal{G}^{\lambda'} \otimes \mathcal{I}_{\beta}) \otimes \mathcal{I}_{\rho'} \ar@{=}[d] \\\mathrm{pr}_1^{*}\mathcal{G}^{\lambda'} \otimes \mathrm{pr}_2^{*}\mathcal{G}^{\lambda'} \otimes \mathcal{I}_{\mathrm{pr}_1^{*}\beta + \mathrm{pr}_2^{*}\beta} \ar[r]_-{\mathcal{M}^{\eta'} \otimes \id_0} & m^{*}\mathcal{G}^{\lambda'} \otimes \mathcal{I}_{\rho'}\otimes \mathcal{I}_{\mathrm{pr}_1^{*}\beta + \mathrm{pr}_2^{*}\beta} }
\end{equation*} 
is connection-preserving. If $(\beta,\varepsilon)$ and $(\beta',\varepsilon')$ are two pairs of such forms, then there exists a 1-form $\xi\in\Omega^1(G,\mathfrak{a})$ with $\Delta\xi=0$, such that $\varepsilon' = \varepsilon + \xi$ and $\beta' = \beta + \mathrm{d}\xi$. 
\end{lemma}

\begin{proof}
We refer the reader to \cite[Section 5.2]{waldorf8} for a general overview about connections on bundle gerbes. 
The connection on  $\mathcal{G}$ consists of two parts: a 2-form $B \in \Omega^2(Y,\mathfrak{a})$, where $\pi\maps Y \to G$ is the surjective submersion of $\mathcal{G}$, and a connection $\omega$ on the principal $\mathcal{B}A$-bundle $P$ over $Y^{[2]}$. The 1-isomorphism $\mathcal{M}$ consists of a principal $\mathcal{B}A$-bundle $Q$ over the fibre product $Z \df (Y \times Y) \lli{m \circ (\pi \times \pi)}\times_{\pi} Y$, and a connection on $\mathcal{M}$ is a connection $\eta$ on $Q$. On $Z$ we have the three projections $\mathrm{pr}_i\maps Z \to Y$, and the surjective submersion $\zeta \df (\pi \times \pi) \circ (\mathrm{pr}_1,\mathrm{pr}_2)\maps  Z \to G\times G$. One of the various conditions that relate all these differential forms is
\begin{equation}
\label{eq:detrho}
\mathrm{pr}_1^{*}B + \mathrm{pr}_2^{*}B
 \eq \mathrm{pr}_3^{*}B + \zeta^{*}\rho + \mathrm{d}\eta\text{;}
 \end{equation}
in particular, $\rho$ is uniquely determined by $B$ and $\eta$. We claim that the remaining conditions imply the following statement: for $(B,\omega,\eta)$ and $(B',\omega',\eta')$ two choices of a multiplicative connection on $(\mathcal{G},\mathcal{M},\alpha)$, there exist  a 2-form $\beta \in \Omega^2(G)$ and a  1-form $\varepsilon \in \Omega^1(Y)$ such that
\begin{equation}
\label{eq:conntrans}
B' \eq B - \pi^{*}\beta + \mathrm{d}\varepsilon
\quomma
\omega' \eq \omega + \mathrm{pr}_2^{*}\varepsilon - \mathrm{pr}_1^{*}\varepsilon
\quand
\eta' \eq \eta + \mathrm{pr}_1^{*}\varepsilon + \mathrm{pr}_2^{*}\varepsilon - \mathrm{pr}_3^{*}\varepsilon\text{.}
\end{equation}
Indeed, by Lemma \cite[Lemma 3.3.5]{waldorf8} the two connections $(B,\omega)$ and $(B',\omega')$ on $\mathcal{G}$ differ by a pair $(\beta,\varepsilon)$ in the way stated above. A priori, $(\beta,\varepsilon)$ are defined up to $\alpha \in \Omega^1(G,\mathfrak{a})$ acting by $(\beta + \mathrm{d}\alpha,\varepsilon + \pi^{*}\alpha)$. 
The connections $\eta$ and $\eta'$ differ a priori by a uniquely defined 1-form $\delta \in \Omega^1(Z,\mathfrak{a})$. The condition that the bundle isomorphism of $\mathcal{M}$ preserves the new as well as the old connections imposes the condition
\begin{equation*}
\zeta_1^{*}(\delta + \mathrm{pr}_{1}^{*}\varepsilon + \mathrm{pr}_2^{*}\varepsilon - \mathrm{pr}_3^{*}\varepsilon) \eq \zeta_2^{*}(\delta + \mathrm{pr}_{1}^{*}\varepsilon + \mathrm{pr}_2^{*}\varepsilon - \mathrm{pr}_3^{*}\varepsilon)
\end{equation*}
over $Z \times_{G \times G} Z$, where $\zeta_1,\zeta_2$ are the two projections. Thus, there exists a unique 1-form $\gamma \in \Omega^1(G \times G,\mathfrak{a})$ such that
\begin{equation*}
\zeta^{*}\gamma\eq\delta + \mathrm{pr}_{1}^{*}\varepsilon + \mathrm{pr}_2^{*}\varepsilon - \mathrm{pr}_3^{*}\varepsilon\text{.} 
\end{equation*}
The condition that the 2-isomorphism $\alpha$ respect both the old and the new connections imposes the condition that $\Delta\gamma \eq 0$, where $\Delta$ is the alternating sum over the pullbacks along the face maps in the simplicial manifold $G^{\bullet}$, forming the complex
\begin{equation*}
\Omega^1(G) \to \Omega^1(G \times G) \to \Omega^1(G \times G \times G) \to \hdots
\end{equation*}
This complex has no cohomology in degree two since $G$ is compact (see the proof of \cite[Proposition 2.3.8]{waldorf5}). Thus, there exists  a 1-form $\kappa \in \Omega^1(G)$ such that $\Delta\kappa \eq \gamma$. Now, the new pair $(\beta + \mathrm{d}\kappa,\varepsilon + \pi^{*}\kappa)$ satisfies \erf{eq:conntrans}. One can now check that the corresponding 2-forms $\rho$ and $\rho'$ determined by \erf{eq:detrho} satisfy the claimed identity.

Now suppose $(\beta,\varepsilon)$ and $(\beta',\varepsilon')$ are two pairs of  forms satisfying\erf{eq:conntrans}. The second equation implies that $\varepsilon'-\varepsilon$ descends to a 1-form $\xi\in\Omega^1(G,\mathfrak{a})$, and the first equation shows that $\beta'-\beta=\mathrm{d}\xi$. The third equation implies $\Delta\xi=0$.
\end{proof}

Using Lemma \ref{lem:1} we get the following.

\begin{proposition}
\label{prop:loopgroupindep}
The central extension $L\mathcal{G}^{\lambda,\eta}$ of Theorem \ref{th:central} is independent of the choice of the multiplicative connection $(\lambda,\eta)$ in the following sense: two multiplicative connections $(\lambda,\eta)$ and $(\lambda',\eta')$ determine a homotopy class of equivalences $L\mathcal{G}^{\lambda,\eta} \cong L\mathcal{G}^{\lambda',\eta'}$. 
\end{proposition}

\begin{proof}
Let $(\lambda',\eta')$ be another multiplicative connection on $(\mathcal{G},\mathcal{M},\alpha)$. By Lemma \ref{lem:1} there exists a connection $\varepsilon$ on the identity $\id\maps  \mathcal{G} \to \mathcal{G}$ and a 2-form $\beta \in \Omega^2(G,\mathfrak{a})$ such that
\begin{equation*}
\id_{\varepsilon}\maps  \mathcal{G}^{\lambda} \to \mathcal{G}^{\lambda'} \otimes \mathcal{I}_{\beta}
\end{equation*}
is a connection-preserving 1-isomorphism. We transgress to the loop space, and use the canonical trivializations $t_{\beta}$ of the transgression of the  bundle gerbe $\mathcal{I}_{\beta}$. We get an isomorphism
\begin{equation}
\label{eq:transideps}
\varphi_{\varepsilon} \df (\id \otimes t_{\beta}) \circ  \mathscr{T}_{\id_{\varepsilon}}: \mathscr{T}_{\mathcal{G}^{\lambda}} \to \mathscr{T}_{\mathcal{G}^{\lambda'}}
\end{equation} 
of principal $\upi_1\Gamma$-bundles over $LM$, and a commutative diagram
\begin{equation*}
\alxydim{@R=1.2cm@C=1.5cm}{ \mathrm{pr}_1^{*}\mathscr{T}_{\mathcal{G}^{\lambda}} \otimes \mathrm{pr}_2^{*}\mathscr{T}_{\mathcal{G}^{\lambda}} \ar[d]_{\mathrm{pr}_1^{*}\varphi_{\varepsilon} \otimes \mathrm{pr}_2^{*}\varphi_{\varepsilon}} \ar[r]^-{\mathscr{T}_{\mathcal{M}^{\eta}}} &  m^{*}\mathscr{T}_{\mathcal{G}^{\lambda}} \ar[d]^{m^{*}\varphi_{\varepsilon}} \\  \mathrm{pr}_1^{*}\mathscr{T}_{\mathcal{G}^{\lambda'}} \otimes \mathrm{pr}_2^{*}\mathscr{T}_{\mathcal{G}^{\lambda'}} \ar[r]_-{\mathscr{T}_{\mathcal{M}^{\eta'}}} & m^{*}\mathscr{T}_{\mathcal{G}^{\lambda'}}\text{,}}
\end{equation*} 
where we have used that the transgression of the identity $\id_0\maps  \mathcal{I}_{\rho} \to \mathcal{I}_{\rho}$ between trivial bundle gerbes obviously exchanges the canonical trivializations. Summarizing,  $\varphi_{\epsilon}$ is an equivalence of central extensions.

In order to see how the maps $\varphi_{\varepsilon}$ and $\varphi_{\varepsilon'}$ differ, let $\mathcal{T}$ be an element in the fibre of $\mathscr{T}_{\mathcal{G}^{\lambda}}$ over a loop $\tau$, i.e. a trivialization $\mathcal{T}: \tau^{*}\mathcal{G}^{\lambda} \to \mathcal{I}_0$. Then, $\varphi_{\varepsilon}(\mathcal{T}) =  (\mathcal{T} \circ \id_{\varepsilon}^{-1}) \otimes \tau^{*}\id$. Since $\id_{\varepsilon'}^{-1} = \id_{\varepsilon}^{-1} \otimes \trivlin_{\xi}$, we have 
\begin{equation*}
\varphi_{\varepsilon'}|_{\tau} = \varphi_{\varepsilon'}|_{\tau} \cdot \exp \left (2 \pi \im \int_{\tau} \xi \right )\text{.}
\end{equation*}
This shows that the equivalence $\varphi_{\varepsilon}$ depends on the choice of $\varepsilon$, and that the equivalences $\varphi_{\varepsilon}$ and $\varphi_{\varepsilon'}$ are related by a smooth map $\tau \mapsto \exp \left (2 \pi \im \int_{\tau} \xi \right )$  which is multiplicative $(\Delta\xi=0)$ and null-homotopic (the $\Delta$-closed forms $\Omega^1(G,\mathfrak{a})$ form a vector space).
\end{proof}

According to Proposition \ref{prop:loopgroupindep} we may write $L\mathcal{G}$ for the central extension of Theorem \ref{th:central} without mentioning the choice of the multiplicative connection on $\mathcal{G}$.

\begin{definition}
\label{def:loopgroup}
Let $\Gamma$ be a smoothly separable Lie 2-group with $\upi_0\Gamma$ compact. The Fréchet Lie group $L\mathcal{G}_{\Gamma}$ is denoted by $L\Gamma$ and called the \emph{loop group} of  $\Gamma$. 
\end{definition}

\begin{example}
For the following two \quot{extremal} examples one can see by just looking at the sequence \erf{eq:ex} what the loop group is.
\begin{enumerate}[(i)]

\item 
If $\Gamma \eq \mathcal{B}A$ for an abelian Lie group $A$, then $L\Gamma \eq A$.

\item
If $\Gamma \eq \idmorph G$ for a Lie group $G$, then $L\Gamma \eq LG$, i.e., $L\Gamma$ is the ordinary loop group. 
\end{enumerate}
\end{example} 

A Lie 2-group homomorphism between Lie 2-groups $\Gamma$ and $\Omega$ is a smooth anafunctor $\Lambda\maps  \Gamma \to \Omega$ which is compatible with the multiplication functors up to a coherent transformation; see \cite[Eq. 2.4.4]{NW11}. We recall from  \cite[Appendix B]{nikolausb} that a Lie 2-group homomorphism induces smooth maps 
\begin{equation*}
\upi_0\Lambda\maps  \upi_0\Gamma \to \upi_0\Omega
\quand
\upi_1\Lambda\maps  \upi_1\Gamma \to \upi_1\Omega\text{.}
\end{equation*}
In order to study the relation between the loop groups $L\Gamma$ and $L\Omega$ we need the following two technical lemmata for preparation.

\begin{lemma}
\label{lem:multconnind}
Let $\mathcal{G}$ and $\mathcal{H}$ be multiplicative bundle gerbes over $G$, let $\mathcal{A}\maps \mathcal{G} \to \mathcal{H}$ be a multiplicative 1-isomorphism, and let $(\lambda,\eta)$ be a multiplicative connection on $\mathcal{G}$. Then, there exist a multiplicative connection $(\lambda',\eta')$ on $\mathcal{H}$, and a connection $\varepsilon$ on $\mathcal{A}$, such that $\mathcal{A}^{\varepsilon}\maps  \mathcal{G}^{\lambda} \to \mathcal{H}^{\lambda'}$ is connection-preserving.
\end{lemma}

\begin{proof}
First of all, there exists a connection $\varepsilon$ on $\mathcal{A}$  and a connection $\lambda'$ on $\mathcal{H}$ such that $\mathcal{A}^{\varepsilon}\maps\mathcal{G}^{\lambda} \to \mathcal{H}^{\lambda'}$ is connection-preserving \cite[Lemma 5.2.4]{waldorf8}. Part of the structure of a \emph{multiplicative} 1-isomorphism is a 2-isomorphism \begin{equation*}
\alxydim{@=1.2cm}{\mathrm{pr}_1^{*}\mathcal{G}^{\lambda} \otimes \mathrm{pr}_2^{*}\mathcal{G}^{\lambda} \ar[d]_{\mathrm{pr}_1^{*}\mathcal{A} \otimes \mathrm{pr}_2^{*}\mathcal{A}} \ar[r]^-{\mathcal{M}} & m^{*}\mathcal{G} ^{\lambda}\otimes \mathcal{I}_{\rho} \ar@{=>}[dl]|*+{\gamma} \ar[d]^{m^{*}\mathcal{A}} \\ \mathrm{pr}_1^{*}\mathcal{H} ^{\lambda'}\otimes \mathrm{pr}_2^{*}\mathcal{H}^{\lambda'} \ar[r]_-{\mathcal{M}'} & m^{*}\mathcal{H} ^{\lambda'}\otimes \mathcal{I}_{\rho}\text{,} }
\end{equation*} 
which goes between 1-isomorphisms between bundle gerbes with connections. The 1-isomorphisms on three sides are in fact equipped with compatible connections. By \cite[Lemma 5.2.5]{waldorf8} there exists a compatible connection $\eta'$ on $\mathcal{M}'$ such that $\gamma$ is connection-preserving. Since the associator $\alpha'$ of $\mathcal{H}$ is uniquely determined by the associator $\alpha$ of $\mathcal{G}$ and $\gamma$, it follows that $\alpha'$ is also connection-preserving. Thus, $(\lambda',\eta')$ is a multiplicative connection on $\mathcal{H}$. 
\end{proof}

\begin{lemma}
\label{lem:multisoindep}
Let $\mathcal{G}$ and $\mathcal{H}$ be isomorphic multiplicative bundle gerbes over $G$. Let $(\lambda,\eta)$ and $(\lambda',\eta')$ be multiplicative connections on $\mathcal{G}$ and $\mathcal{H}$, respectively. Then, there exists an isomorphism $L\mathcal{G}^{\lambda,\eta} \cong L\mathcal{H}^{\lambda',\eta'}$.
\end{lemma}

\begin{proof}
By Lemma \ref{lem:multconnind} there exists a connection $(\tilde\lambda,\tilde\eta)$ on $\mathcal{H}$ such that $\mathcal{G}^{\lambda}$ and $\mathcal{H}^{\tilde\lambda}$ are isomorphic as multiplicative gerbes. This isomorphism transgresses to a Lie group isomorphism $L\mathcal{G}^{\lambda,\eta} \cong L\mathcal{H}^{\tilde\lambda,\tilde\eta}$. Together with Proposition \ref{prop:loopgroupindep}, this shows the claim. \end{proof}

\begin{proposition}
\label{prop:loopgroupcov}
If $\Lambda\maps  \Gamma \to \Omega$ is a Lie 2-group homomorphism, there exists a Lie group homomorphism $L\Lambda\maps  L\Gamma \to L\Omega$ such that the diagram
\begin{equation*}
\alxydim{@=1.2cm}{\upi_1\Gamma \ar[d]_{\upi_1\Lambda} \ar[r] & L\Gamma \ar[d]^{L\Lambda} \ar[r] & L\upi_0\Gamma \ar[d]^{L\upi_0\Lambda} \\ \upi_1\Omega \ar[r] & L\Omega \ar[r] & L\upi_0\Omega }
\end{equation*}
is commutative. 
\end{proposition}

\begin{proof}
We recall that the  bundle gerbes $\mathcal{G}_{\Gamma}$ and $\mathcal{G}_{\Omega}$ correspond to the  principal 2-bundles $\Gamma$ over $\upi_0\Gamma$ and $\Omega$ over $\upi_0\Omega$, respectively, under the equivalence of Theorem \ref{thm:equiv}. The Lie 2-group homomorphism $\Lambda$ defines a  1-morphism
\begin{equation*}
(\upi_1\Lambda)_{*}\Gamma \to (\upi_0\Lambda)^{*}\Omega
\end{equation*}
of $\mathcal{B}\upi_1\Omega$-2-bundles over $\upi_0\Gamma$, which is compatible with the composition of $\Gamma$ and $\Omega$ in exactly such a way that  the  induced 1-isomorphism
\begin{equation}
\label{eq:multiso}
\mathcal{D}\maps (\upi_1\Lambda)_{*}\mathcal{G}_\Gamma \to (\upi_0\Lambda)^{*}\mathcal{G}_\Omega
\end{equation}
 of  $\mathcal{B}\upi_1\Omega$-bundle gerbes over $\upi_0\Gamma$ is multiplicative. Now let $(\lambda,\eta)$ and $(\lambda',\eta')$ be multiplicative connections on $\mathcal{G}_{\Gamma}$ and $\mathcal{G}_{\Omega}$, respectively. These induce connections on $(\upi_1\Lambda)_{*}\mathcal{G}_\Gamma$ and $(\upi_0\Lambda)^{*}\mathcal{G}_\Omega$. Now Lemma \ref{lem:multisoindep} applies and yields an isomorphism  $L((\upi_1\Lambda)_{*}\mathcal{G}_\Gamma)^{(\upi_1\Lambda)_{*}(\lambda,\eta)}\cong L((\upi_0\Lambda)^{*}\mathcal{G}_\Omega)^{(\upi_0\Lambda)^{*}(\lambda',\eta')}$. Since the transgression functor $\mathscr{T}$ is contravariant in the base manifold and covariant in the Lie group, we get an isomorphism
\begin{equation*}
(\upi_1\Lambda)_{*}(L\mathcal{G}_\Gamma^{\lambda,\eta}) \cong L((\upi_1\Lambda)_{*}\mathcal{G}_\Gamma)^{(\upi_1\Lambda)_{*}(\lambda,\eta)}\cong L((\upi_0\Lambda)^{*}\mathcal{G}_\Omega)^{(\upi_0\Lambda)^{*}(\lambda',\eta')}\cong (\upi_0\Lambda)^{*}(L\mathcal{G}_\Omega^{\lambda',\eta'})\text{.}
\end{equation*}
Such an isomorphism is the same as the claimed homomorphism $L\Lambda$. 
\end{proof}

\begin{remark}
The group homomorphism $L\Lambda$ of Proposition \ref{prop:loopgroupcov} is determined up to homotopy, just like the equivalence of Proposition \ref{prop:loopgroupindep}. 
\end{remark}

\subsection{Transgression of the Lifting Bundle 2-Gerbe}

 According to Theorem \ref{thm:lift}, a $\Gamma$-bundle gerbe $\mathcal{G}$ defines a principal $\mathcal{B}\upi_0\Gamma$-bundle $E \df \pi_{*}(\mathcal{G})$ over $M$ together with a trivialization $\mathbb{T}_{\mathcal{G}}$ of the lifting bundle 2-gerbe $\mathbb{L}_E$.    
The crucial point is that although $\mathcal{G}$ is a non-abelian gerbe, the lifting bundle 2-gerbe $\mathbb{L}_{E}$ and its trivialization $\mathbb{T}_{\mathcal{G}}$ are both abelian.

First we remark that $LE$ is a principal $L\upi_0\Gamma$-bundle over $LM$. Here we need the assumption that $\upi_0\Gamma$ is connected \cite[Lemma 5.1]{waldorf13}. The obstruction to lift the structure group $LE$ from $L\upi_0\Gamma$ to the loop group $L\Gamma$ is represented by the lifting bundle gerbe $\mathcal{L}_{LE,L\Gamma}$ over $LM$. In short, there is a canonical equivalence of groupoids
\begin{equation}
\label{eq:liftingtheory}
\triv(\mathcal{L}_{LE,L\Gamma}) \cong \lift_{LE,L\Gamma}
\end{equation}
which is an analogue of Theorem \ref{thm:lift} for ordinary bundles \cite{murray,waldorf13}.
The plan is to transgress the lifting bundle 2-gerbe $\mathbb{L}_E$ to the loop space and identify this transgression $\mathscr{T}_{\mathbb{L}_E}$ with the lifting gerbe. The main result of this section is the following.

\begin{proposition}
\label{prop:translifting}
Let $\Gamma$ be a smoothly separable Lie 2-group with $\upi_0\Gamma$ compact and connected, and let $E$ be a principal $\upi_0\Gamma$-bundle over $M$. \begin{enumerate}[(a)]

\item
There is a canonical isomorphism
\begin{equation*}
\phi_{E,\Gamma}\maps  \mathscr{T}_{\mathbb{L}_E} \to \mathcal{L}_{LE,L\Gamma}
\end{equation*}
between the transgression of the lifting bundle 2-gerbe $\mathbb{L}_E$ and the lifting bundle gerbe for the problem of lifting the structure group of $LE$ from $LG$ to $L\Gamma$. 

\item
Suppose $\Omega$ is another Lie 2-group, and $\Lambda\maps \Gamma\to\Omega$ is a Lie 2-group homomorphism. Let $L\Lambda\maps L\Gamma \to L\Omega$ be the Lie group homomorphism of Proposition \ref{prop:loopgroupindep}, let $E' \df (\upi_0\Lambda)_{*}(E)$, and let $\mathcal{L}_{L\Lambda}\maps  \mathcal{L}_{LE,L\Gamma} \to \mathcal{L}_{LE',L\Omega}$ be the induced isomorphism between lifting bundle gerbes. Then, there exists an isomorphism
\begin{equation*}
\mathbb{C}\maps  (\upi_1\Lambda)_{*}(\mathbb{L}_E) \to \mathbb{L}_{E'}
\end{equation*}
of $\mathcal{B}\upi_1\Omega$-bundle 2-gerbes over $M$ and a 2-isomorphism
\begin{equation*}
\alxydim{@C=2.4cm@R=1.2cm}{(\upi_1\Lambda)_{*} (\mathscr{T}_{\mathbb{L}_E^{}}) \ar[d]_{\mathscr{T}_{\mathbb{C}}} \ar[r]^-{(\upi_1\Lambda)_{*} (\phi_{E,\Gamma})} & (\upi_1\Lambda)_{*} (\mathcal{L}_{LE,L\Gamma}) \ar@{=>}[dl] \ar[d]^{\mathcal{L}_{L\Lambda}} \\ \mathscr{T}_{\mathbb{L}_{E'}} \ar[r]_-{\phi_{E',\Omega}} & \mathcal{L}_{LE',L\Omega}}
\end{equation*}
of $\mathcal{B}\upi_1\Omega$-bundle gerbes over $LM$.
In other words, the canonical isomorphism of (a) is compatible with Lie 2-group homomorphisms.
\end{enumerate}

\end{proposition}

Proposition \ref{prop:translifting} is proved by the following lemmata, in which we carefully deal with the dependence of $L\Gamma$ and $\mathscr{T}_{\mathbb{L}_E}$ on choices of connections. For this discussion, we assume  a connected compact Lie group $G$, an abelian Lie group $A$, and some multiplicative $\mathcal{B}A$-bundle gerbe $(\mathcal{G},\mathcal{M},\alpha)$ over $G$. Let $\mathbb{L}_E$ be the corresponding lifting bundle 2-gerbe. We equip it with a connection following \cite[Section 3.2]{waldorf5}. Let $(\lambda,\eta)$ be a multiplicative connection on $\mathcal{G}$ with associated 2-form $\rho$. Then, using the exact sequence of \cite[Section 8]{murray} there exist a 2-form $\omega \in \Omega^2(E^{[2]})$ with $\Delta\omega \eq- \delta_2^{*}\rho$ and a 3-form $C \in \Omega^{3}(E)$ with $\Delta C \eq \delta_1^{*}\mathrm{curv}(\mathcal{G}^{\lambda}) + \mathrm{d}\omega$. 

We equip the bundle gerbe $\delta_1^{*}\mathcal{G}$ over $E^{[2]}$ with the connection $\delta_1^{*}\lambda + \omega$. For simplicity, we shall denote the resulting bundle gerbe with connection by $\mathcal{H}_{\lambda,\omega} \df \delta_1^{*}\mathcal{G}^{\lambda} \otimes \mathcal{I}_{\omega}$. 
It follows that  the 1-isomorphism $\mathcal{N}_\eta$ defined by
\begin{equation*}
\alxydim{@R=1.2cm}{\mathrm{pr}_{23}^{*}\mathcal{H}_{\lambda,\omega} \otimes \mathrm{pr}_{12}^{*}\mathcal{H}_{\lambda,\omega} \ar@{=}[r] & \mathrm{pr}_{23}^{*}(\delta_1^{*}\mathcal{G}^{\lambda}) \otimes \mathrm{pr}_{12}^{*}(\delta_1^{*}\mathcal{G}^{\lambda}) \otimes \mathcal{I}_{\mathrm{pr}_{12}^{*}\omega + \mathrm{pr}_{23}^{*}\omega} \ar[d]^-{\delta_2^{*}\mathcal{M}_{\eta} \otimes \id} \\ & \mathrm{pr}_{13}^{*}(\delta_1^{*}\mathcal{G}^{\lambda} \otimes \mathcal{I}_{\rho})  \otimes \mathcal{I}_{\mathrm{pr}_{12}^{*}\omega + \mathrm{pr}_{23}^{*}\omega} \ar@{=}[r] &  \mathrm{pr}_{13}^{*}\mathcal{H}_{\lambda,\omega}}
\end{equation*}
is connection-preserving, and that the 2-isomorphism $\delta_3^{*}\alpha \otimes \id$ is also connection-preserving. This means that $\chi \df (C,\delta_1^{*}\lambda + \omega,\delta_2^{*}\eta )$ is a connection on $\mathbb{L}_E$, together denoted by $\mathbb{L}^{\chi}_{E}$. We say that the pair $(C,\omega)$ is an \emph{extension} of the multiplicative connection $(\lambda,\eta)$ to a connection $\chi$ on $\mathbb{L}_E$.
Given the connection $\chi$, we define the following $\mathcal{B}A$-bundle gerbe $\mathscr{T}_{\mathbb{L}_E^{\chi}}$ over $LM$:
\begin{enumerate}[(a)]
\item 
its surjective submersion is $LE \to LM$. Note that $(LE)^{[k]} \eq L(E^{[k]})$.

\item
its principal $\mathcal{B}A$-bundle over $LE^{[2]}$ is $\mathscr{T}_{\mathcal{H}_{\lambda,\omega}}$.  

\item
its bundle gerbe product is $\mathscr{T}_{\mathcal{N}_{\eta}}$. 
Its associativity is guaranteed by the 2-isomorphism $\delta_3^{*}\alpha$.
 
\end{enumerate}
We remark that we have not used the 3-form $C$; it is so far only included for completeness. We have the following \quot{lifting commutes with transgression} result.

\begin{lemma}
\label{lem:2}
There  is a canonical isomorphism
$\phi_{\omega} \maps  \mathscr{T}_{\mathbb{L}_E^{\chi}} \to \mathcal{L}_{LE,L\mathcal{G}^{\lambda,\eta}}$
of $A$-bundle gerbes over $LM$. 
\end{lemma}

\begin{proof}
Both bundle gerbes have the same surjective submersion, $LE \to LM$.
The claimed isomorphism comes from an isomorphism 
\begin{equation*}
   \phi_{\omega}\maps \mathscr{T}_{\mathcal{H}_{\lambda,\omega}} \to \delta_1^{*}L\mathcal{G}^{\lambda,\eta}
\end{equation*}
of principal $A$-bundles over $LE^{[2]}$, where $\delta_1\maps  LE^{[2]} \to LG$ is the difference map analogous to the one used in Definition \ref{def:lifting2gerbe} and $\delta_1^{*}L\mathcal{G}^{\lambda,\eta}$ is the principal bundle of $\mathcal{L}_{LE}$. Indeed, since $\mathcal{H}_{\lambda,\omega}\eq\delta_1^{*}\mathcal{G}^{\lambda} \nobr\otimes\nobr \mathcal{I}_{\omega}$ by definition, the  isomorphism $\phi_{\omega}$ is given by the canonical trivialization $t_{\omega}$ of $\mathscr{T}_{\mathcal{I}_{\omega}}$. It remains to ensure that $\phi_{\omega}$   is compatible with the bundle gerbes products:  $\mathscr{T}_{\mathcal{N}^{\eta}}$ on $\mathscr{T}_{\mathbb{L}_E^{\chi}}$ and $\delta_2^{*}\mathscr{T}_{\mathcal{M}^{\eta}}$ on $\mathcal{L}_{LE, L\mathcal{G}^{\lambda,\eta}}$. This follows   immediately from the definition of $\mathcal{N}^{\eta}$.
\end{proof}

The following lemma investigates the dependence of the isomorphism of Lemma \ref{lem:2} under a change of connections.  We suppose that $(\lambda,\eta)$ and $(\lambda',\eta')$ are multiplicative connections on $\mathcal{G}$, and that $(C,\omega)$ and $(C',\omega')$  are extensions to connections $\chi$ and $\chi'$ on $\mathbb{L}_E$, respectively. We assume that
\begin{equation*}
\varphi\maps  L\mathcal{G}^{\lambda,\eta} \to L\mathcal{G}^{\lambda',\eta'}
\end{equation*}
is one of the equivalences of central extensions of Proposition \ref{prop:loopgroupindep}, and remark that it induces a 1-isomorphism
\begin{equation*}
\mathcal{L}_{\varphi}\maps  \mathcal{L}_{LE,L\mathcal{G}^{\lambda,\eta}} \to \mathcal{L}_{LE,L\mathcal{G}^{\lambda',\eta'}}
\end{equation*}  
of lifting bundle gerbes over $LM$.

\begin{lemma}
\label{lem:3}
There exists a 3-form $F \in \Omega^3(M,\mathfrak{a})$, a connection $\kappa$ on the identity isomorphism $\id\maps  \mathbb{L}_E \to \mathbb{L}_E$ such that
\begin{equation*}
\id^{\kappa}\maps  \mathbb{L}_E^{\chi} \to \mathbb{L}_E^{\chi'} \otimes \mathbb{I}_F
\end{equation*}
is a connection-preserving isomorphism between $A$-bundle 2-gerbes over $M$, and a 2-isomorphism
\begin{equation*}
\alxydim{@=1.2cm}{\mathscr{T}_{\mathbb{L}_E^{\chi}} \ar[d]_{\mathscr{T}_{\id^{\kappa}}} \ar[r]^-{\phi_{\omega}} & \mathcal{L}_{LE,L\mathcal{G}^{\lambda,\eta}} \ar@{=>}[dl] \ar[d]^{\mathcal{L}_{\varphi}} \\ \mathscr{T}_{\mathbb{L}_E^{\chi'}} \ar[r]_-{\phi_{\omega'}} & \mathcal{L}_{LE,L\mathcal{G}^{\lambda',\eta'}}}
\end{equation*}
between bundle gerbe isomorphisms over $LM$.
\end{lemma}

\begin{proof}
We let $\id_{\varepsilon}\maps  \mathcal{G}^{\lambda} \to \mathcal{G}^{\lambda'} \otimes \mathcal{I}_{\beta}$ be a connection-preserving 1-isomorphism as in Lemma \ref{lem:1}, so that $\varphi \eq t_{\beta} \circ \mathscr{T}_{\id_{\varepsilon}}$. We recall that $\rho' + \Delta\beta \eq \rho$, and calculate that $\Delta(\omega'-\omega + \delta_2^{*}\beta)\eq0$. Thus, there exists $\kappa \in \Omega^2(E,\mathfrak{a})$ such that 
\begin{equation}
\label{eq:kappa}
\Delta\kappa \eq \omega'-\omega + \delta_2^{*}\beta\text{.}
\end{equation}
We use $\kappa$ to construct the 1-isomorphism $\id^{\kappa}$. 
It consists of the trivial bundle gerbe $\mathcal{I}_{\kappa}$ over $E$. We define $F$ such that
\begin{equation*}
\mathrm{d}\kappa \eq \pi^{*}F + C' - C\text{,} 
\end{equation*}
 which is the required compatibility condition for the 3-curvings. 
The 1-isomorphism $\id_{\kappa}$ consists further of the connection-preserving 1-isomorphism
\begin{equation*}
\alxydim{@C=2cm}{\mathrm{pr}_1^{*}\mathcal{I}_{\kappa} \otimes \mathcal{H}_{\lambda',\omega'} \eq \mathcal{I}_{\mathrm{pr}_1^{*}\kappa + \omega'} \otimes \delta_1^{*}\mathcal{G}^{\lambda'} \ar[r]^-{\id \otimes \delta_1^{*}\mathcal{\id}_{\varepsilon}} & \mathcal{I}_{\mathrm{pr}_1^{*}\kappa + \omega'} \otimes \delta_1^{*}\mathcal{G}^{\lambda'} \otimes \mathcal{I}_{\delta_1^{*}\beta} \eq \mathcal{H}_{\lambda,\omega} \otimes \mathrm{pr}_2^{*}\mathcal{I}_{\kappa}
\text{.}}
\end{equation*}
The latter satisfies the higher coherence conditions because of the commutative diagram in Lemma \ref{lem:1}. The 2-isomorphism is given by the trivialization $t_{\kappa}$ of $\mathscr{T}_{\mathcal{I}_{\kappa}}$; \erf{eq:kappa} provides the necessary compatibility relation.
\end{proof}

Now we return to a smoothly separable Lie 2-group $\Gamma$ with $\upi_0\Gamma$ compact and connected, and look at the situation where $G \df \upi_0\Gamma$, $A \df \upi_1\Gamma$, and $\mathcal{G} \df \mathcal{G}_{\Gamma}$. Then, Lemma \ref{lem:2} yields the isomorphism $\phi_{E,\Gamma}$ claimed in Proposition \ref{prop:translifting} (a), and Lemma \ref{lem:3} ensures that it is well-defined under different choices of connections.
Next we consider  a Lie 2-group homomorphism $\Lambda\maps \Gamma \to \Omega$, and write $\widetilde{\mathcal{G}} \df (\upi_1\Lambda)_{*}(\mathcal{G}_\Gamma)$. Let 
\begin{equation}
\label{eq:multiso2}
\mathcal{D}\maps \widetilde{\mathcal{G}} \to (\upi_0\Lambda)^{*}\mathcal{G}_\Omega
\end{equation}
be the 1-isomorphism  \erf{eq:multiso} defined by $\Lambda$. 
Let $(\lambda',\eta')$ be a connection on $\mathcal{G}_{\Omega}$. By Lemma \ref{lem:multconnind} there exists a multiplicative connection $(\tilde\lambda,\tilde\eta)$ on $\widetilde{\mathcal{G}}$
and a connection $\varepsilon$ on $\mathcal{D}$ such that $\mathcal{D}^{\varepsilon}$ is connection-preserving. 
Let $E$ be a principal $\upi_0\Gamma$-bundle over $M$, and let $E' \df (\upi_0\Lambda)_{*}(E)$. There is a canonical map $f\maps  E \to E'$ which is equivariant along $\upi_0\Lambda$. This can be rephrased as the commutativity of the diagram
\begin{equation}
\label{eq:diffcomm}
\alxydim{@=1.2cm}{E^{[n+1]} \ar[d]_{f^{n+1}} \ar[r]^{\delta_n} & \upi_0\Gamma^{n} \ar[d]^{\upi_0\Lambda} \\  E'^{[n+1]} \ar[r]_{\delta_n} & \upi_0\Omega^n}
\end{equation}
for all $n$. We get an induced isomorphism 
\begin{equation*}
\mathcal{L}_{\mathscr{T}_{\mathcal{D}^{\varepsilon}}}\maps   \mathcal{L}_{LE,L\widetilde{\mathcal{G}}^{\tilde\lambda,\tilde\eta}} \to \mathcal{L}_{LE',L\mathcal{G}_{\Omega}^{\lambda',\eta'}} 
\end{equation*} 
Let $(C',\omega')$ be an extension of $(\lambda',\eta')$ to a connection $\chi'$ on $\mathbb{L}_{E'}$. Then, $\tilde C \df (\pi_0\Lambda)^{*}C'$ and $\tilde \omega \df (\pi_0\Lambda)^{*}\omega'$ is an extension of $(\tilde\lambda,\tilde\eta)$ to a connection $\tilde\chi$ on $\widetilde{\mathbb{L}_E}$, the lifting bundle 2-gerbe formed by the $\upi_0\Gamma$-bundle $E$ and the $\mathcal{B}\upi_1\Omega$-bundle gerbe $\widetilde{\mathcal{G}}$ over $\upi_0\Gamma$. The commutativity of \erf{eq:diffcomm} implies that the 1-isomorphism $\mathcal{D}^{\varepsilon}$  induces a connection-preserving 1-isomorphism
\begin{equation*}
\mathbb{L}_{\mathcal{D}^{\varepsilon}}\maps  \widetilde{\mathbb{L}_E}^{\tilde\chi} \to \mathbb{L}_{E'}^{\chi'}\text{.}
\end{equation*}
It is straightforward to check that the diagram
\begin{equation}
\label{eq:lem4}
\alxydim{@C=1.8cm@R=1.2cm}{\mathscr{T}_{\widetilde{\mathbb{L}_E}^{\tilde\chi}} \ar[d]_{\mathscr{T}_{\mathbb{L}_{\mathcal{D}^{\varepsilon}}}} \ar[r]^-{\phi_{\tilde \omega}} & \mathcal{L}_{LE,L\widetilde{\mathcal{G}}^{\tilde\lambda,\tilde\eta}} \ar@{=>}[dl] \ar[d]^{\mathcal{L}_{\mathscr{T}_{\mathcal{D}^{\varepsilon}}}} \\ \mathscr{T}_{\mathbb{L}_{E'}^{\chi'}} \ar[r]_{\phi_{\omega'}} & \mathcal{L}_{LE',L\mathcal{G}_{\Omega}^{\lambda',\eta'}}}
\end{equation} 
of 1-isomorphisms 
between $\mathcal{B}\upi_1\Omega$-bundle gerbes over $LM$ is strictly commutative.

\begin{lemma}
\label{lem:5}
Let $\Lambda\maps \Gamma \to \Omega$ be a Lie 2-group homomorphism, let $(\lambda,\eta)$ and $(\lambda',\eta')$ be connections on $\mathcal{G}_{\Gamma}$ and $\mathcal{G}_{\Omega}$, and let $\chi$ and $\chi'$ be extensions to connections on $\mathbb{L}_E$ and $\mathbb{L}_{E'}$, respectively. Then, there exists a 3-form $F \in \Omega^3(M,\mathfrak{a})$, a connection-preserving isomorphism
\begin{equation*}
\mathbb{C}^{\rho}\maps  (\upi_1\Lambda)_{*} (\mathbb{L}_E^{\chi}) \to \mathbb{L}_E^{\chi'} \otimes \mathbb{I}_F\text{,}
\end{equation*}
and a 2-isomorphism
\begin{equation*}
\alxydim{@C=1.8cm@R=1.2cm}{(\upi_1\Lambda)_{*} (\mathscr{T}_{\mathbb{L}_E^{\chi}}) \ar[d]_{\mathscr{T}_{\mathbb{C}^{\rho}}} \ar[r]^-{(\upi_1\Lambda)_{*} (\phi_{\omega})} & (\upi_1\Lambda)_{*} (\mathcal{L}_{LE,L\mathcal{G}_{\Gamma}^{\lambda,\eta}}) \ar@{=>}[dl] \ar[d]^{\mathcal{L}_{\varphi}} \\ \mathscr{T}_{\mathbb{L}_{E'}^{\chi'}} \ar[r]_-{\phi_{\omega'}} & \mathcal{L}_{LE,L\mathcal{G}_{\Omega}^{\lambda',\eta'}}}
\end{equation*}
between bundle gerbe isomorphisms over $LM$. 
\end{lemma}

\begin{proof}
On the multiplicative bundle gerbe $\widetilde{\mathcal{G}} \eq (\upi_1\Lambda)_{*}(\mathcal{G}_{\Gamma})$ we have two connections: $(\tilde\lambda,\tilde\eta)$ and $(\upi_1\Lambda)_{*}(\lambda,\eta)$. Accordingly, on $\widetilde{\mathbb{L}_E} \eq (\upi_1\Lambda)_{*}(\mathbb{L}_E)$ we have the two extensions $(\tilde C,\tilde\omega)$ and $(\upi_1\Lambda)_{*}(C,\omega)$ to connections $\tilde\chi$ and $(\upi_1\Lambda)_{*}\chi$. Thus, Lemma \ref{lem:3} applies, and provides a 2-isomorphism
\begin{equation*}
\alxydim{@C=2cm@R=1.2cm}{(\upi_1\Lambda)_{*} (\mathscr{T}_{\mathbb{L}_E^{\chi}}) \ar[d]_{\mathscr{T}_{\id^{\kappa}}} \ar[r]^-{(\upi_1\Lambda)_{*}(\phi_{\omega})} & (\upi_1\Lambda)_{*} (\mathcal{L}_{LE,L\mathcal{G}_{\Gamma}^{\lambda,\eta}}) \ar@{=>}[dl] \ar[d]^{\mathcal{L}_{\varphi}} \\ \mathscr{T}_{\widetilde{\mathbb{L}_E}^{\tilde\chi}} \ar[r]_-{\phi_{\tilde\omega}} & \mathcal{L}_{LE,L\widetilde{\mathcal{G}}^{\tilde\lambda,\tilde\eta}}\text{.}}
\end{equation*}
Now the commutative diagram \erf{eq:lem4} extends this 2-isomorphism to the claimed one.
\end{proof}

Lemma \ref{lem:5} proves Proposition \ref{prop:translifting} (b).

\subsection{Transgression of Trivializations}

Now we come to the trivialization $\mathbb{T}_{\mathcal{G}}$ of $\mathbb{L}_{E}$ associated to the $\Gamma$-bundle gerbe $\mathcal{G}$. Since we have equipped $\mathbb{L}_E$ with a connection $\chi$, it follows that there exists a compatible connection $\rho$ on the trivialization $\mathbb{T}_{\mathcal{G}}$ \cite[Proposition 3.3.1]{waldorf8}. If $\mathbb{T}_{\mathcal{G}}$ consists of a $\mathcal{B}\upi_1\Gamma$-bundle gerbe $\mathcal{S}$ over $E$, of a 1-isomorphism 
\begin{equation*}
\mathcal{C}\maps  \mathrm{pr}_1^{*}\mathcal{S} \otimes \delta_1^{*}\mathcal{G}_{\Gamma} \to \mathrm{pr}_2^{*}\mathcal{S}
\end{equation*}
over $E^{[2]}$, and of a 2-isomorphism $\zeta$ over $E^{[3]}$, the connection $\rho$ is a pair $\rho\eq(\gamma,\nu)$ of a connection $\gamma$ on $\mathcal{S}$ and of a connection $\nu$ on $\mathcal{C}$, such that $\mathcal{C}$ and $\zeta$ are connection-preserving.
As described in \cite[Section 4.2]{waldorf8}, the trivialization $\mathbb{T}_{\mathcal{G}}^{\rho}$ with connection can be transgressed to a trivialization $\mathscr{T}_{\mathbb{T}_{\mathcal{G}}^{\rho}}$ of $\mathscr{T}_{\mathbb{L}_E^{\chi}}$. It consists of:
\begin{enumerate}[(i)]

\item 
the principal $\mathcal{B}\upi_1\Gamma$-bundle $\mathscr{T}_{\mathcal{S}^{\gamma}}$ over $LE$,

\item
the bundle isomorphism
\begin{equation*}
\mathscr{T}_{\mathcal{C}^{\nu}}\maps  \mathrm{pr}_1^{*}\mathscr{T}_{\mathcal{S}^{\gamma}} \otimes \mathscr{T}_{\mathcal{H}_{\lambda,\omega}} \to \mathrm{pr}_2^{*}\mathscr{T}_{\mathcal{S}^{\gamma}}
\end{equation*}
over $LE^{[2]}$, which is compatible with the bundle gerbe product $\mathscr{T}_{\mathcal{N}^{\eta}}$ due to the existence of the  2-isomorphism $\zeta$.

\end{enumerate} 
Under the canonical identification of  Proposition \ref{prop:translifting} and the equivalence \erf{eq:liftingtheory},  $\mathscr{T}_{\mathbb{T}_{\mathcal{G}}^{\rho}}$  determines a  principal $\mathcal{B}L\mathcal{G}_{\Gamma}^{\lambda,\eta}$-bundle over $LM$, which we denote by $\mathscr{T}_{\mathcal{G}}^{\chi,\rho}$.

\begin{lemma}
\label{lem:6}
The $\mathcal{B}L\mathcal{G}^{\lambda,\eta}$-bundle $\mathscr{T}_{\mathcal{G}}^{\chi,\rho}$ over $LM$ is independent of the choice of the connection $\rho$ up to bundle isomorphisms.
\end{lemma}

\begin{proof}
We may regard the trivialization as an isomorphism $\mathbb{T}_{\mathcal{G}}^{\rho}\maps  \mathbb{L}_E^{\chi} \to \mathbb{I}_H$, for some $H \in \Omega^3(M)$. Another connection  $\rho'$ corresponds to another isomorphism $\mathbb{T}_{\mathcal{G}}^{\rho'}\maps  \mathbb{L}_E^{\chi} \to \mathbb{I}_{H'}$. The difference between $\mathbb{T}_{\mathcal{G}}^{\rho}$ and $\mathbb{T}_{\mathcal{G}}^{\rho'}$ can be compensated in a connection on the identity isomorphism $\id\maps \mathbb{I}\to \mathbb{I}$. Up to 2-isomorphisms, such a connection is given by a 2-form $\xi$ with $\mathrm{d}\xi \eq H'-H$. Then, there exists a 2-isomorphism
\begin{equation*}
\alxydim{@R=1.2cm}{& \mathbb{L}_E^{\chi} \ar[dl]_{\mathbb{T}_{\mathcal{G}}^{\rho}} \ar[dr]^{\mathbb{T}_{\mathcal{G}}^{\rho'}}^{}="1" \ar@{=>}[dl];"1" & \\ \mathbb{I}_{H} \ar[rr]_{\id^{\xi}} && \mathbb{I}_{H'}\text{.}}
\end{equation*}
Since $\id^{\xi}$ transgresses to the identity between trivial gerbes over $LM$, the 2-isomorphism transgresses to an isomorphism between trivializations of $\mathscr{T}_{\mathbb{L}_E^{\chi}}$.
\end{proof}

\begin{lemma}
\label{lem:7}
Let $\Lambda\maps \Gamma \to \Omega$ be a Lie 2-group homomorphism, let $(\lambda,\eta)$ and $(\lambda',\eta')$ be connections on $\mathcal{G}_{\Gamma}$ and $\mathcal{G}_{\Omega}$, and let $\chi$ and $\chi'$ be extensions to connections on $\mathbb{L}_E$ and $\mathbb{L}_{E'}$, respectively. Let $\rho$ and $\rho'$ be connections on $\mathbb{T}_{\mathcal{G}}$ compatible with $\chi$ and $\chi'$, respectively. Then, there exists an isomorphism 
\begin{equation*}
(L\Lambda)_{*}(\mathscr{T}_{\mathcal{G}}^{\chi,\rho}) \cong \mathscr{T}_{\Lambda_{*}(\mathcal{G})}^{\chi',\rho'}\text{.}
\end{equation*}
\end{lemma}

\begin{proof}
Let
\begin{equation*}
\mathbb{C}^{\rho}\maps  (\upi_1\Lambda)_{*} (\mathbb{L}_E^{\chi}) \to \mathbb{L}_E^{\chi'} \otimes \mathbb{I}_F\text{,}
\end{equation*}
be the isomorphism of Lemma \ref{lem:5}. With the arguments of Lemma \ref{lem:6} there exists a 2-isomorphism
\begin{equation*}
\alxydim{@C=2cm@R=1.2cm}{ (\upi_1\Lambda)_{*} (\mathbb{L}_E^{\chi}) \ar[r]^-{(\upi_1\Lambda)_{*}(\mathbb{T}_{\mathcal{G}}^{\rho})} \ar[d]_{\mathbb{C}^{\rho}}  & (\upi_1\Lambda)_{*}(\mathbb{I}_{H}) \ar@{=>}[dl] \ar[d]^{\id^{\xi}} \\ \mathbb{L}_{E'}^{\chi'} \otimes \mathbb{I}_F \ar[r]_-{\mathbb{T}_{\Lambda_{*}\mathcal{G}}^{\rho'} \otimes \id} & \mathbb{I}_{H'+F}}
\end{equation*}
Its transgression, together with Lemma \ref{lem:6}, yields the claim.
\end{proof}

Notice that Lemma \ref{lem:7} proves for $\Lambda \eq \id_{\Gamma}$ that $\mathscr{T}_{\mathcal{G}}^{\chi,\rho}$ is independent of the choices of all connections, up to bundle isomorphisms. We may hence denote it simply by $\mathscr{T}_{\mathcal{G}}$. 
 
\begin{definition}
Let $\Gamma$ be a smoothly separable Lie 2-group with $\upi_0\Gamma$ compact and connected, and let $\mathcal{G}$ be a $\Gamma$-bundle gerbe  over $M$. Then, the principal $L\Gamma$-bundle $\mathscr{T}_{\mathcal{G}}$ over $LM$ is called the \emph{transgression} of $\mathcal{G}$.
\end{definition}

Summarizing the results collected above, we have the following.

\begin{theorem}
\label{th:trans}
Let $\Gamma$ be a smoothly separable Lie 2-group with $\upi_0\Gamma$ compact and connected. Then, the assignment $\mathcal{G}\mapsto \mathscr{T}_{\mathcal{G}}$ defines a  map
\begin{equation*}
\mathscr{T}\maps  \h^1(M,\Gamma) \to \h^0(LM,\mathcal{B}L\Gamma)
\end{equation*}
with the following properties:
\begin{enumerate}[(i)]
\item 
it is contravariant in $M$ and covariant in $\Gamma$,

\item
for $\Gamma\eq\mathcal{B}A$, it reduces to the ordinary transgression homomorphism
\begin{equation*}
\tau\maps  \h^1(M,\mathcal{B}A) \to \h^0(LM,\mathcal{B}A)\text{,}
\end{equation*}

\item
for $\Gamma \eq \idmorph{G}$, it reduces to the looping of bundles
\begin{equation*}
L\maps  \h^1(M,\idmorph{G}) \to \h^0(LM,\mathcal{B}LG)\text{.} 
\end{equation*}
\end{enumerate}
\end{theorem}

\begin{proof}
The well-definedness of the map $\mathscr{T}$ as well as the covariance in (i) follow from Lemma \ref{lem:7}. The contravariance in $M$ is evident. 
In (ii) we have, in the notation used above, $E\eq M$ and correspondingly $LE\eq LM$. This means that both the  lifting bundle 2-gerbe $\mathbb{L}_E$ and the lifting gerbe $\mathcal{L}_{LE,L\Gamma}$ are canonically trivial. Under these canonical identifications, we can identify $\mathcal{G}\eq \mathcal{S}$, where $\mathcal{S}$ is the bundle gerbe in $\mathbb{T}_{\mathcal{G}}$. Also, we can identify $\mathscr{T}_{\mathcal{S}}$ with $\mathscr{T}_{\mathcal{G}}$. But $\mathscr{T}_{\mathcal{S}}$ is the ordinary, abelian transgression which underlies the homomorphism $\tau$. In (iii), we have $A\eq*$, so that both lifting problems are trivial. In particular, $LE$ is the lift of the structure group of $LE$ from $LG$ to $LG$, i.e., $LE \eq \mathscr{T}_{\mathcal{G}}$. 
\end{proof}

For completeness, we include the following corollary.

\begin{corollary}
The following diagram is commutative:
\begin{equation*}
\alxydim{@C=0.55cm@R=1.2cm}{
\h^0(M,\idmorph{\upi_0\Gamma}) \ar[r]\ar[d]^L & \h^1(M,B\upi_1\Gamma) \ar[r]\ar[d]^\tau &  \h^1(M,\Gamma) \ar[r]\ar[d]^{\mathscr{T}} & \h^1(M,\idmorph{\upi_0\Gamma}) \ar[r]\ar[d]^L & \h^2(M,B\upi_1\Gamma)\ar[d]^\tau \\
\h^0(LM,\idmorph{L\upi_0\Gamma}) \ar[r] & \h^0(LM,\mathcal{B}\upi_1\Gamma) \ar[r] &  \h^0(LM,\mathcal{B}L\Gamma) \ar[r] & \h^0(LM,\mathcal{B}L\upi_0\Gamma) \ar[r] & \h^1(LM,\mathcal{B}\upi_1\Gamma)
}
\end{equation*}
\end{corollary}

\begin{proof}
The two diagrams in the middle are commutative because of Theorem \ref{th:trans}. The commutativity of the outer diagrams is a statement in ordinary \v Cech cohomology and straightforward to verify. 
\end{proof}

\setsecnumdepth 1

\section{Application to String Structures}

\label{sec:string}

We recall that $\str n$ is a topological group defined up to homotopy equivalence by requiring that it is a 3-connected cover of $\spin n$.  It is known that $\str n$ cannot be realized as a finite-dimensional Lie group, but as a infinite-dimensional Fréchet Lie group  \cite{nikolausb}. For several reasons, however, it is more attractive to work with \emph{Lie 2-group} models for $\str n$. 
\begin{definition}[{{\cite[Definition 4.10]{nikolausb}}}]
\label{def:string}
A \emph{2-group model for $\str n$} is a smoothly separable Lie 2-group $\Gamma$ such that
\begin{equation*}
\upi_0\Gamma \eq \spin n
\quand
\upi_1\Gamma \eq \ueins\text{,}
\end{equation*}
and such that the geometric realization $|\Gamma|$ has the homotopy type of $\str n$.
\end{definition}

The first 2-group model for $\str n$ has been constructed in \cite{baez9} using central extensions of loop groups. Another model has been provided in \cite{nikolausb} based on the above-mentioned Fréchet Lie group realization of $\str n$. A further construction  appears in \cite{waldorf14}. We remark that  the constructions of \cite{Henriques2008,pries2} are not 2-group models in the sense of Definition \ref{def:string}, since they are not \emph{strict} 2-groups.

The major motivation to look at the group $\str n$ comes from string theory; in particular, from fermionic sigma models.
2-group models for  $\str n$ are so attractive because they lead directly to (non-abelian) gerbes, which are in turn intimately related to string theory. Non-abelian gerbes for 2-group models for $\str n$ have been considered in  \cite{stevenson3,jurco1}, and can  be treated with the theory developed in \cite{NW11} and the present article.

In the following we describe an application of the lifting theory developed in Section \ref{sec:lift2gerbe} to string structures. Let $M$ be a  spin manifold of dimension $n$, i.e., the structure group of the frame bundle $FM$ of $M$ is lifted to $\spin n$. Topologically, a \emph{string structure on $M$} is a further lift of the structure group of $FM$ to $\str n$. Homotopy theory shows that the obstruction against this further lift is a certain class $\frac{1}{2}p_1(M) \in \mathrm{H}^4(M,\Z)$, and that  equivalence classes of string structures form a torsor over $\mathrm{H}^3(M,\Z)$. 

The topological definition of a string structure has an evident 2-group-counterpart.

\begin{definition}[{{\cite{stevenson3,jurco1}}}]
\label{def:ss1}
Suppose $\Gamma$ is a 2-group model for $\str n$. Then, a \emph{string structure on $M$} is a $\Gamma$-lift of $FM$ in the sense of Definition \ref{def:lift}, i.e., a $\Gamma$-bundle gerbe $\mathcal{S}$ over $M$ together with an isomorphism
$\varphi\maps  \pi_{*}(\mathcal{S}) \to FM$
of $\spin n$-bundles over $M$.
\end{definition}

String structures in the sense of Definition \ref{def:ss1} form a bigroupoid $\lift_{\Gamma}(FM)$. Via the equivalence between non-abelian gerbes and classifying maps \cite[Section 4]{NW11} one can easily show that Definition \ref{def:ss1} is a refinement of the topological notion of a string structure.

Another way to define string structures is to look at the obstruction class $\frac{1}{2}p_1(M)$. It can be represented by a Chern-Simons 2-gerbe $\mathbb{CS}_{E}(\mathcal{G})$ \cite[Theorem 1.1.3]{waldorf8}. We recall (also see Remark \ref{rem:cs}) that a Chern-Simons 2-gerbe receives as input data a principal $G$-bundle $E$ over $M$ and a multiplicative $\mathcal{B}\ueins$-bundle gerbe $\mathcal{G}$ over $G$. Here, with $G \eq \spin n$, the $G$-bundle $E$ is the frame bundle $FM$, and $\mathcal{G}$ can be any multiplicative $\mathcal{B}\ueins$-bundle gerbe over $\spin n$ with level one, i.e., characteristic class $[\mathcal{G}] \eq 1 \in \Z \eq \mathrm{H}^3(\spin n,\Z)$.  Now, the idea of the following definition is that a string structure is a trivialization of the obstruction against string structures.

\begin{definition}[{{\cite[Definition 1.1.5]{waldorf8}}}]
\label{def:ss2}
Let $\mathcal{G}$ be a multiplicative $\mathcal{B}\ueins$-bundle gerbe over $\spin n$ with level one. Then, a \emph{string structure on $M$} is a trivialization of the Chern-Simons 2-gerbe $\mathbb{CS}_{FM}(\mathcal{G})$. \end{definition}

String structures in the sense of Definition \ref{def:ss2} form a bigroupoid $\triv (\mathbb{CS}_{FM}(\mathcal{G}))$. We remark that a priori no string group or 2-group model for $\str n$ is involved in Definition \ref{def:ss2}. However, it depends on the input  of the multiplicative bundle gerbe $\mathcal{G}$.

As explained in \cite{waldorf8,waldorf14} there is a canonical way to produce such a multiplicative $\mathcal{B}\ueins$-bundle gerbe over $\spin n$ with level one: one starts with the basic gerbe $\mathcal{G}_{bas}$ over $\mathcal{B}\ueins$, which enjoys a finite-dimensional, Lie-theoretical construction \cite{meinrenken1,gawedzki1}. The multiplicative structure can be obtained by a transgression-regression procedure \cite{waldorf14}.

Another method to obtain the multiplicative bundle gerbe $\mathcal{G}$ is to start with a 2-group model for $\str n$. We infer from \cite[Remark 4.11]{nikolausb} the following. 
\begin{lemma}
\label{lem:levelone}
If $\Gamma$ is a 2-group model for $\str n$, then the multiplicative $\mathcal{B}\ueins$-bundle gerbe $\mathcal{G}_{\Gamma}$ introduced in Section \ref{sec:lift2gerbe} has level one. 
\end{lemma}

As noted in Remark \ref{rem:cs}, we have the coincidence
\begin{equation*}
\mathbb{CS}_{FM}(\mathcal{G}_{\Gamma}) \eq \mathbb{L}_{FM}\text{,}
\end{equation*}
i.e., the Chern-Simons 2-gerbe \emph{is} the lifting bundle 2-gerbe for the problem of lifting the structure group of $FM$ from $\spin n$ to $\Gamma$. Now, Theorem \ref{thm:lift} becomes as follows.
\begin{theorem}
\label{th:coinc}
The two notions of string structures from  Definition \ref{def:ss1} and Definition \ref{def:ss2} coincide. More precisely, for any 2-group model $\Gamma$ for $\str n$ there is an equivalence of bigroupoids:
\begin{equation*}
\triv(\mathbb{CS}_{FM}(\mathcal{G}_{\Gamma})) \cong \lift_{\Gamma}(FM) \text{.}
\end{equation*}
\end{theorem}

A third (inequivalent) definition is  to say that a string structure on $M$ is the same as a \emph{spin structure} on the free loop space $LM$ \cite{mclaughlin1}. Here we require $n>1$ in order to make $\spin n$ connected. 

\begin{definition}
A \emph{spin structure on $LM$} is a lift of the structure group of the looped frame bundle $LFM$ from $L\spin n$ to the universal central extension
\begin{equation*}
1 \to \ueins \to \widehat{L\spin n} \to L\spin n \to 1\text{.}
\end{equation*}
\end{definition}

The relation between string structures on $M$ and spin structures on $LM$ is based on the following fact concerning the loop group of a 2-group model for the string group, see Definition \ref{def:loopgroup}.
\begin{lemma}
\label{lem:loopstring}
Let $\Gamma$ be a 2-group model for $\str n$. Then, $L\Gamma \cong \widehat{L\spin n}$. 
\end{lemma} 

\begin{proof}
By Lemma \ref{lem:levelone} the multiplicative bundle gerbe $\mathcal{G}_{\Gamma}$ has level one. By \cite[Corollary 3.1.9]{waldorf5}, its transgression $\mathscr{T}_{\mathcal{G}_{\Gamma}} \eq\maps  L\Gamma$ is the universal central extension of $L\spin n$.
\end{proof}

Let us first recall how string structures in the sense of Definition \ref{def:ss2} induce a spin structure on $LM$. For this purpose, we simply reduce the procedure described in Section \ref{sec:transgression} to the case where $\Gamma$ is a 2-group model for $\str n$, reproducing a description given in \cite{waldorf2009}. We choose a multiplicative connection on $\mathcal{G}_{\Gamma}$, and a connection on the frame bundle $FM$, for instance the Levi-Cevita connection. By Proposition \ref{prop:translifting}, the transgression $\mathscr{T}_{\mathbb{CS}_{FM}(\mathcal{G})}$ is the lifting bundle gerbe for the problem of lifting the structure group of $LFM$ from $L\spin n$ to $L\Gamma$. 

If now $\mathbb{T}$ is a trivialization of $\mathbb{CS}_{FM}(\mathcal{G}_{\Gamma})$, it admits a connection compatible with the connection on $\mathbb{CS}_{FM}(\mathcal{G}_{\Gamma})$, and transgresses to a trivialization $\mathscr{T}_{\mathbb{T}}$ of $\mathscr{T}_{\mathbb{CS}_{FM}(\mathcal{G}_{\Gamma})}$, which is precisely a spin structure on $LM$. Lemma \ref{lem:7} shows that we get a well-defined map
\begin{equation}
\label{eq:trans1}
\bigset{4.3cm}{Isomorphism classes of trivializations of $\mathbb{CS}_{FM}(\mathcal{G}_{\Gamma})$} \to 
\bigset{3.5cm}{Isomorphism classes of spin structures on $LM$} \text{.}
\end{equation}

\begin{remark}
It is not possible to upgrade this map to a functor between groupoids, because connections cannot be chosen in a functorial way. However, one can include the connections into the structure on both hand sides, and so obtain a functor between a groupoid of \emph{geometric string structures} on $M$ and a groupoid of \emph{geometric spin structures} on $LM$.
\end{remark} 

Now suppose we have a string structure in the sense of Definition \ref{def:ss1}, i.e., a $\Gamma$-bundle gerbe $\mathcal{S}$ together with a bundle morphism $\varphi \maps  \pi_{*}(\mathcal{S}) \to FM$. In non-abelian cohomology, this is a class $[\mathcal{S}] \in \h^1(M,\Gamma)$ such that $\pi_{*}([\mathcal{G}]) \eq [FM] \in \h^0(M,\mathcal{B}\spin n)$. Applying the transgression map
\begin{equation*}
\mathscr{T}\maps  \h^1(M,\Gamma) \to \h^0(LM,\mathcal{B}\widehat{L\spin n})
\end{equation*}
of Theorem \ref{th:trans} produces a class $\mathscr{T}([\mathcal{S}]) \in \h^0(LM,\mathcal{B}\widehat{L\spin n})$. Theorem \ref{th:trans} (ii) and (iv) imply that the extension of this class is $[LFM] \in \h^0(LM,\mathcal{B}\spin n)$, i.e., $\mathscr{T}([\mathcal{S}])$ is an isomorphism class of spin structures on $LM$. Summarizing, we have a map
\begin{equation}
\label{eq:trans2}
\bigset{3.1cm}{Isomorphism classes of $\Gamma$-lifts of $FM$} \to 
\bigset{3.5cm}{Isomorphism classes of spin structures on $LM$} \text{.}
\end{equation}

The two maps \erf{eq:trans1} and \erf{eq:trans2} are compatible with the equivalence of Theorem \ref{th:coinc} in the following sense.

\begin{theorem}
\label{th:stringtrans}
There is a commutative diagram:
\begin{equation*}
\alxydim{@C=-0.5cm@R=1.8cm}{\left \lbrace\;\txt{Isomorphism classes of\\string structures on $M$\\ in the sense of Definition \ref{def:ss2}\\\\}\; \right \rbrace \ar[dr]_{\erf{eq:trans1}} \ar@<0.1cm>[rr]^-{\text{Theorem \ref{th:coinc}}} && \ar@<0.1cm>[ll] \left \lbrace\;\txt{Isomorphism classes of\\string structures on $M$\\ in the sense of Definition \ref{def:ss1}\\\\}\; \right \rbrace \ar[dl]^{\erf{eq:trans2}} \\ &\left \lbrace\;\txt{Isomorphism classes of\\spin structures on $LM$\\\\}\; \right \rbrace&}
\end{equation*}
\end{theorem}

\begin{proof}
Parsing through the constructions, the map \erf{eq:trans2} (transgression in non-abelian cohomology) was defined by passing to the associated trivialization of the lifting 2-gerbe and then transgressing in that abelian setting. The latter procedure is the composition of the inverse of the bijection of Theorem \ref{th:coinc} with the map \erf{eq:trans1}; thus, the diagram is commutative. 
\end{proof}

\kobib{.}

\end{document}